\documentclass[11pt, reqno]{amsart}

\usepackage{amsmath,amsthm,amssymb,mathrsfs}
\usepackage[abbrev,non-sorted-cites]{amsrefs}
\usepackage{xcolor,graphicx}
\usepackage{verbatim}
\usepackage{datetime}
\usepackage{hyperref}

\theoremstyle{plain}
    \newtheorem{thm}{Theorem}[section]
    \newtheorem{lem}[thm]{Lemma}
    \newtheorem{prop}[thm]{Proposition}
    \newtheorem{cor}[thm]{Corollary} 
\theoremstyle{definition}
    \newtheorem{defn}[thm]{Definition}
    \newtheorem{assumption}[thm]{Assumption}
    \newtheorem{rmk}[thm]{Remark}
    
    \newtheorem*{ack*}{Acknowledgement}

\numberwithin{equation}{section}

\newcommand\ip[2]{\langle{#1},{#2}\rangle}

\newcommand\pl{\partial}

\newcommand\w{\wedge}

\newcommand\dt{\delta}
\newcommand\ep{\epsilon}
\newcommand\vep{\varepsilon}

\newcommand\om{\omega}

\newcommand\gm{\gamma}
\newcommand\kp{\kappa}

\newcommand\bt{\beta}

\newcommand\ld{\lambda}

\newcommand\Om{\Omega}

\newcommand\Ld{\Lambda}

\newcommand\Dt{\Delta}

\newcommand\CE{\mathcal{E}}

\newcommand\CU{\mathcal{U}}

\newcommand\BC{\mathbb{C}}

\newcommand\BR{\mathbb{R}}
\newcommand\BN{\mathbb{N}}
\newcommand\BZ{\mathbb{Z}}

\newcommand\fa{\mathfrak{a}}

\newcommand\mr{\mathring}
\newcommand\ul{\underline}
\newcommand\td{\tilde}
\newcommand\br{\bar}

\newcommand\bft{\mathbf{t}}
\newcommand\bfn{\mathbf{n}}

\newcommand\bfI{\mathbf{I}}

\newcommand\dd{{\mathrm d}}

\newcommand\phixaxis{\phi_{\text{ax}}}
\newcommand\phireg{\phi_{\text{reg}}}
\newcommand\msR{\mathscr{R}}
\newcommand\msW{\mathscr{W}}
\newcommand\msS{\mathscr{S}}
\newcommand\msL{\mathscr{L}}
\newcommand\msZ{\mathscr{Z}}
\newcommand\mass{\mathfrak{m}}

\DeclareMathOperator{\pr}{\mathsf{pr}}

\DeclareMathOperator{\Sym}{Sym}
\DeclareMathOperator{\dist}{dist}

\begin{document}

\title[Infinite-Time Singularities for Lagrangian Mean Curvature Flow]{Infinite-Time Singularities with Vanishing Mean Curvature for Lagrangian Mean Curvature Flow in Gibbons--Hawking Spaces}
\subjclass{Primary 53E10, Secondary 53C26, 53C38}

\thanks{C.-J.~Tsai is supported in part by the National Science and Technology Council grant 112-2628-M-002-004-MY4.  Part of this work was carried out when P.-H.~Lee was visiting the National Center for Theoretical Sciences.
}

\author{Ping-Hung Lee}
\address{Department of Mathematics, Columbia University, New York, NY 10027, USA}
\email{pl2975@columbia.edu}
\author{Chung-Jun Tsai}
\address{Department of Mathematics, National Taiwan University, and National Center for Theoretical Sciences, Math Division, Taipei 10617, Taiwan}
\email{cjtsai@ntu.edu.tw}


\begin{abstract}
We construct infinite-time singularities with vanishing mean curvature for Lagrangian mean curvature flow in Gibbons--Hawking spaces.  We consider circle-invariant Lagrangian \(2\)-spheres whose quotient curves are concave and are \(C^2\)-close to a collection of consecutive collinear segments.  We prove that the corresponding flow exists smoothly for all time and converges to the associated \(A_{n-1}\)-chain of special Lagrangian spheres.  Although the mean curvature converges uniformly to zero, the second fundamental form becomes unbounded.  More precisely, \(\log\max |A(\,\cdot\,,t)|\) is comparable to \(\sqrt{t}\) as \(t\to\infty\).  The proof is based on a one-parameter family of barrier curves and a detailed analysis of their asymptotics.  In this way, we refine the infinite-time convergence picture arising in the work of Lotay and Oliveira by proving curvature blow-up and estimating its rate in this semi-stable case.
\end{abstract}

\maketitle
\section{Introduction}

A special Lagrangian submanifold in a Calabi--Yau manifold is calibrated by the real part of the holomorphic volume form, and is therefore volume minimizing in its homology class \cite{Harvey-Lawson-82}.   Since the Lagrangian condition is preserved by mean curvature flow \cite{Smoczyk-96}, it is natural to ask whether the flow can be used to deform a Lagrangian submanifold to a special Lagrangian representative, or more generally to decompose it into special Lagrangian pieces.  This is the point of view behind the conjectures of Thomas--Yau \cites{Thomas-01,TY-02} and Joyce \cite{Joyce-15}.   As shown by Neves \cite{Neves-13}, for any Lagrangian one can always find arbitrarily small Hamiltonian perturbations from which the mean curvature flow develops finite-time singularities.  Therefore, an interesting aspect in this picture is that singularities may occur both at finite and infinite times.

In \cite{Lotay-Oliveira-1}, Lotay and Oliveira studied special Lagrangians and Lagrangian mean curvature flow in hyperk\"ahler four-manifolds with circle symmetry.  In the circle-invariant setting, they proved the Thomas conjecture.  Under the almost-calibrated assumption, they also proved a version of the Thomas--Yau conjecture.  A key ingredient is that a circle-invariant Lagrangian surface corresponds to a planar curve in the quotient, and Lagrangian mean curvature flow reduces to a modified curve shortening flow.

In \cite{Lotay-Oliveira-2}, Lotay and Oliveira further developed this reduction to analyze neck-pinch singularities and the flow through singularities.  They proved that, for compact embedded almost-calibrated circle-invariant Lagrangians, the flow can be continued through a finite number of finite-time neck-pinch singularities and converges, in the sense of currents, to an \(A_{n-1}\)-chain of special Lagrangian spheres.  This verifies several central features of Joyce's conjectural picture in this symmetric setting.
The works of Lotay and Oliveira draw on results concerning singularity formation of Lagrangian mean curvature flow due to Lambert--Lotay--Schulze \cite{LLS-21} and Lotay--Schulze--Sz\'ekelyhidi \cite{LSS-24}.  More recently, Sz\'ekelyhidi \cite{Gabor-26} studied the stability of the singularity formation.

{For a finite-time singularity, the blow-up rate of the mean curvature can be derived from \cite{Lotay-Oliveira-2}*{Proposition 4.2}, provided\footnote{They proved \(\phi^{-\frac12}\cdot|H|\to0\) for a finite time singularity.} one knows the blow-up rate of the Gibbons--Hawking potential \(\phi\).}  The main interest of this paper is to investigate the blow-up rate of the mean curvature for the infinite-time singularity. 
In particular, we consider the case when the components of the limiting \(A_{n-1}\)-chain of special Lagrangian spheres in \cite{Lotay-Oliveira-2}*{Theorem 1.4} have the same phase.  In this case, this chain is represented in the quotient by a collection of consecutive, collinear segments (see Remark~1.5 of \cite{Lotay-Oliveira-2}).  It turns out that the infinite-time singularity has vanishing mean curvature:

\begin{thm} \label{thm:intro}
    Let \(L\) be a circle-invariant, semi-stable Lagrangian \(2\)-sphere in a Gibbons--Hawking space. Suppose that its quotient curve is concave and \(C^2\)-close to the consecutive collinear segments corresponding to an \(A_{n-1}\)-chain of special Lagrangian spheres, where \(n\geq3\).  Then its Lagrangian mean curvature flow exists smoothly for all time and converges to the \(A_{n-1}\)-chain of special Lagrangian spheres as \(t\to\infty\).  Moreover, \(H(\,\cdot\,,t)\to 0\) uniformly as \(t\to\infty\), while
    \[ 0 < \liminf_{t\to\infty}\frac{\log\max|A(\,\cdot\,,t)|}{\sqrt{t}} \leq \limsup_{t\to\infty}\frac{\log\max|A(\,\cdot\,,t)|}{\sqrt{t}} < \infty ~. \]
\end{thm}

The precise statement is in Theorem~\ref{thm:main}.  The fact that it forms an infinite-time singularity has already been proved in \cite{Lotay-Oliveira-2}*{proof of Corollary~5.5}.  In our setting, we give another proof using the upper and lower barriers, which are also used to estimate the mean curvature and the second fundamental form.

For general (not necessarily Lagrangian) mean curvature flow, the behavior of the mean curvature at a finite-time singularity is an important issue.  Much work has been devoted to this question; see \cite{Stolarski-23} and the references therein.  In particular, Stolarski \cite{Stolarski-23} proved that certain finite-time singularities can have bounded mean curvature.  In Theorem~\ref{thm:intro}, the mean curvature does not merely remain bounded; it tends to zero uniformly.  To our knowledge, this provides the first example of an infinite-time singularity of mean curvature flow for which the mean curvature converges uniformly to zero.  This implies that for the mean curvature flow, long-time existence and vanishing mean curvature are not sufficient to conclude smooth convergence as \(t\to\infty\).

Recently, several infinite-time singularities of mean curvature flow have been constructed.  In \cites{CS-25,CS-26}, Chen and Sun constructed immortal mean curvature flows with multiplicity-two convergence.  In \cite{STW-24}, Su, Tsai and Wood constructed Lagrangian mean curvature flows which exist for all time and converge to an immersed special Lagrangian, with a polynomial blow-up rate for the second fundamental form.  It is interesting to see whether the mean curvature in these works remains bounded.

The paper is organized as follows.  Section~\ref{sec:hk-and-Lag} recalls the Gibbons--Hawking ansatz and the circle-invariant Lagrangian setting.  Section~\ref{sec:LMCF} derives the graphical form of the Lagrangian mean curvature flow and records the comparison principle used later.  Most of the results in Section~\ref{sec:hk-and-Lag} and Section~\ref{subsec:flow-basic} are taken from \cites{Lotay-Oliveira-1,Lotay-Oliveira-2} and are included for completeness.  Section~\ref{sec:barrier} constructs the barrier functions and proves their basic properties.  Section~\ref{sec:asymptotic} studies the asymptotic behavior of the barrier as the parameter goes to \(0\).  Section~\ref{sec:main-argument} proves the main theorem by comparing the flow with the barrier family.

{\it Conventions}. Throughout the paper, constants such as \(c_1,c_2,\ldots\) are local to the section in which they appear; constants with the same name in different sections are not assumed to be related, unless explicitly stated otherwise.  For a function \(u(x,t)\), we write \(u'(x,t)\) for its partial derivative with respect to \(x\).

\begin{ack*}
    The authors are grateful to Wei-Bo Su and Albert Wood for many helpful discussions.
\end{ack*}

\section{The Gibbons--Hawking Ansatz} \label{sec:hk-and-Lag}

In this section, we recall the background of the Gibbons--Hawking ansatz needed for this paper.  Since our focus is on singularity formation, throughout most of the paper, we work on a proper open subset of an ALE or ALF hyperk\"ahler \(4\)-manifold admitting a tri-Hamiltonian circle action.  We refer the reader to \cite{Lotay-Oliveira-1}*{Section~2}, \cite{Lotay-Oliveira-2}*{Section~2} and the references therein for more about the Gibbons--Hawking ansatz.

\subsection{The Hyperk\"ahler Structure}
    The ambient space \(\CU\) is a non-compact, hyperk\"ahler \(4\)-manifold admitting a tri-Hamiltonian circle action.  The action has fixed points \(p_1,\cdots,p_n\) with \(n\geq 3\), and is free on \(\CU_0 := \CU\setminus\{p_1,\cdots,p_n\}\).  The quotient space \(\Om\) is an open neighborhood of \([-1,1]\times\{0\}\times\{0\}\) in \(\BR^3\).  The quotient map \(\pr:\CU\to\Om\) maps \(p_i\) to \((a_i,0,0)\), where \(-1 = a_1 < a_2 < \cdots < a_n = 1\).

    There is a positive, smooth function \(\phi(x,y,z)\) on \(\Om\setminus\{(a_1,0,0),\cdots,(a_n,0,0)\}\) such that
    \begin{align} \label{fct:phi-regular-part}
        \phireg(x,y,z) := \phi(x,y,z) - \frac{1}{2}\sum_{i=1}^n\frac{1}{\sqrt{(x-a_i)^2+y^2+z^2}}
    \end{align}
    is non-negative, smooth and harmonic on \(\Om\).  Here, harmonicity is defined by using the standard metric on \(\BR^3\).  By shrinking \(\Om\) if necessary, we may assume that \(\phireg\) and its derivatives are bounded on \(\Om\).

    Denote by \(*\) the Hodge star operator of the standard metric on \(\BR^3\).  Note that \(\CU_0\stackrel{\pr}{\to}\Om\setminus\{(a_1,0,0),\cdots,(a_n,0,0)\}\) is a principal circle bundle.  It has a connection \(1\)-form \(\fa\) with \(\dd\fa = \pr^*(*\dd\phi)\).
    On the open dense set \(\CU_0\), the hyperk\"ahler metric is
    \begin{align} \label{eqn:hk-metric}
        \frac{1}{\phi}\fa^2 + \phi\,(\dd x^2 + \dd y^2 + \dd z^2) ~.
    \end{align}
    We will often omit \(\pr^*\).  Near each \(p_i\), the map \(\pr\) is locally modeled on the Hopf map from \(\BC^2\) to \(\BR^3\).  The manifold is oriented by \(\phi\,\fa\w\dd x\w\dd y\w\dd z\), and its hyperk\"ahler triples are
    \begin{align*}
        \fa\w\dd x + \phi\,\dd y\w\dd z ~,\quad \fa\w\dd y + \phi\,\dd z\w\dd x ~,\quad \fa\w\dd z + \phi\,\dd x\w\dd y ~.
    \end{align*}

\subsection{Circle-Invariant Lagrangian Submanifolds}

Consider the Calabi--Yau structure with the K\"ahler form \[\om = \fa\w\dd z + \phi\,\dd x\w\dd y\] and the holomorphic volume form \[\Om = (\fa+\sqrt{-1}\phi\dd z)\w(\dd x+\sqrt{-1}\dd y)~.\]

A \(2\)-dimensional submanifold \(L\) of \(\CU\) is called a Lagrangian if \(\om\) vanishes on \(L\).  From these expressions, one finds that circle-invariant Lagrangians are in one-to-one correspondence with curves in planes parallel to the \(xy\)-plane.  Unless otherwise stated, the Lagrangians will always be assumed to be smooth.  This imposes regularity conditions on the plane curve, and we will address one of them in Lemma~\ref{lem:boundedness-kappa}.  In this setting, a Lagrangian submanifold is said to be special Lagrangian if its mean curvature vanishes; see \cite{Harvey-Lawson-82}*{Section~III} for more details.

For a circle-invariant Lagrangian \(L\), let \(\gm = \pr(L)\).  Parametrize \(\gm\) by the arc-length \(s\) with respect to the flat metric.  By choosing a suitable orientation, \(\frac{\Om|_L}{\text{vol}_L}\) is a unit-complex-valued function that coincides (under \(\pr\)) with \((\dd x + \sqrt{-1}\dd y)(\frac{\dd}{\dd s}\gm)\).  Its argument is the so-called Lagrangian angle, and is well-defined up to \(2\pi\BZ\).  Denote the Lagrangian angle by \(\bt:\gm\to\BR/2\pi\BZ\):
\begin{align}
    e^{\sqrt{-1}\bt} &= (\dd x+\sqrt{-1}\dd y)(\frac{\dd}{\dd s}\gm) = \frac{\Om|_L}{\text{vol}_L} ~.
\end{align}

According to \cite{Lotay-Oliveira-1}*{Proposition~4.6}, the second fundamental form and mean curvature of \(L\) can be expressed in terms of the curve \(\gm\) and the function \(\phi\).
\begin{lem} \label{lem:second-fform-mean-curvature}
    Let \(\bft = \frac{\dd}{\dd s}\gm\) and \(\bfn = \pl_z\times\frac{\dd}{\dd s}\gm\).  Denote the curvature of the plane curve by \(\kp := \ip{\frac{\dd}{\dd s}\bft}{\bfn}_{\BR^3} = \frac{\dd}{\dd s}\bt\).  The norms of the second fundamental form and the mean curvature of \(L\) satisfy the following estimates:
    \begin{align}
        |A|^2 &\leq \frac{1}{\phi} \kp^2 + \frac{3}{\phi^{3}}\bigl[ (\dd\phi(\bfn))^2 + (\dd\phi(\pl_z))^2 \bigr] \quad\text{and} \label{eqn:second-fform} \\
        |H|^2 &= \frac{1}{\phi} \kp^2 ~. \label{eqn:mean-curvature}
    \end{align}
\end{lem}

From \eqref{eqn:mean-curvature}, the pre-image of a straight-line segment in \(\BR^3\) is a {special} Lagrangian.  For \([-1,1]\times\{0\}\times\{0\}\), its pre-image is an \(A_{n-1}\)-chain of {special} Lagrangian spheres, and \(p_i\) (for \(1<i<n\)) is a transverse intersection point of two adjacent spheres.  This \(A_{n-1}\)-chain of {special} Lagrangian spheres will be the infinite-time limit of the Lagrangian mean curvature flow considered in this paper.

For a circle-invariant Lagrangian \(2\)-sphere passing through \(p_1\) and \(p_n\), its image under \(\pr\) is a curve in the \(xy\)-plane connecting \((a_1,0) = (-1,0)\) to \((a_n,0) = (1,0)\) and avoiding the intermediate points \((a_2,0),\cdots,(a_{n-1},0)\).
For such a Lagrangian \(2\)-sphere, the angle \(\bt\) can be chosen to be a smooth real-valued function; see \cite{Lotay-Oliveira-1}*{Example~5.5}.  We now examine the behavior of \(\kp\) (the curvature of the plane curve) near \((a_1,0)\) and \((a_n,0)\).
Let \(r\) be the \(\BR^2\)-distance to \((a_1,0)\).  On the \(xy\)-plane, \(\phi = \frac{1}{2r} + O(1)\) as \(r\to0\).  One infers from the boundedness of \(|H|^2\) and \eqref{eqn:mean-curvature} that \(\kp = O(r^{-\frac{1}{2}})\) as \(r\to0\).  In fact, \(\kp\) cannot blow up as \(r\to 0\).  The underlying reason is that \(p_i\) are coordinate singularities of the metric \eqref{eqn:hk-metric}.

\begin{lem} \label{lem:boundedness-kappa}
    For a smooth, circle-invariant Lagrangian \(2\)-sphere \(L\) passing through \(p_1\) and \(p_n\), the plane curve \(\gm = \pr(L)\) has bounded curvature.
\end{lem}

\begin{proof}
It suffices to show that \(\kp\) is bounded near \((a_1,0)\).  By \eqref{eqn:mean-curvature}, it is equivalent to \(|H| = O(r^{\frac{1}{2}})\) as \(r\to0\).
According to \cite{Harvey-Lawson-82}*{Section~III.2.D}, \(|H|^2 = |\dd \bt|^2\), where the norms are computed by using the hyperk\"ahler metric on \(\CU\).  In this setting, the mean curvature vector must be circle-invariant.  However, at \(T_{p_1}\CU\), the only circle-invariant vector is the zero vector, and \(H(p_1) = 0 = \dd\bt(p_1)\).  It follows from Taylor's theorem on \(\dd\bt\) that for \(p\in L\) close to \(p_1\), \( |\dd\bt(p)| \lesssim \dist_L(p,p_1) \).  From \eqref{eqn:hk-metric}, it is not hard to see that \(\dist_L(p,p_1)\) is comparable to \(\sqrt{r(p)}\), and this lemma follows.
\end{proof}

\begin{rmk} \label{rmk:end-point-tangent}
    Since $\frac{\dd}{\dd s}\beta=\kappa$, $\bt$ is Lipschitz in $s$.  It follows that \(\bft\) has a limit at the endpoints \((a_1,0)\) and \((a_n,0)\).
\end{rmk}

\section{Lagrangian Mean Curvature Flow} \label{sec:LMCF}

\subsection{The Evolution Equation and Existence Criterion} \label{subsec:flow-basic}

We continue with the setting introduced in Section~\ref{sec:hk-and-Lag}.
\begin{assumption} \label{asmp:LMCF}
    Suppose that \(L_0\) is a circle-invariant Lagrangian sphere in \(\CU\) passing through \(p_1\) and \(p_n\), whose mean curvature flow always stays within \(\CU\).  Denote by  \(\{L_t\}_{0\leq t}\) the Lagrangian mean curvature flow with initial condition \(L_0\).

    Since the mean curvature flow preserves the Lagrangian condition and the circle invariance, \(\gm(\,\cdot\,,t) := \pr(L_t)\) is a curve on the \(xy\)-plane and determines \(L_t\).  Since each \(L_t\) is also a sphere, each \(\gm(\,\cdot\,,t)\) must be a curve connecting \((a_1,0)\) and \((a_n,0)\).  It follows that the mean curvature of \(L_t\) vanishes at \(p_1\) and \(p_n\), which can also be seen from the proof of Lemma~\ref{lem:boundedness-kappa}.
\end{assumption}

Here is the evolution equation and the existence criterion.
\begin{lem}[\cite{Lotay-Oliveira-1}*{Proposition~4.5 and Lemma~6.3} and \cite{Lotay-Oliveira-2}*{Proposition~4.1}] \label{lem:LMCF-curve-basic}
    Under Assumption~\ref{asmp:LMCF}, parametrize the curve \(\gm(\,\cdot\,,t)\) by the Euclidean arc-length \(s\).  Then, the family of curves satisfies
    \begin{align} \label{eqn:LMCF-curve}
        \frac{\pl}{\pl t}\gm &= \frac{1}{\phi}\frac{\dd^2\gm}{\dd s^2} = \frac{1}{\phi}\kp\bfn  ~.
    \end{align}
    Moreover, the flow exists as long as \eqref{eqn:second-fform} is bounded, and the curves do not reach \((a_i,0)\) for \(i=2,\ldots,n-1\).  The boundedness of \eqref{eqn:second-fform} is equivalent to the boundedness of \(\phi^{-\frac{1}{2}}\kp\), \(\phi^{-\frac{1}{2}}\dd\log\phi(\bfn)\) and \(\phi^{-\frac{1}{2}}\dd\log\phi(\pl_z)\).
\end{lem}

Lemma~\ref{lem:LMCF-curve-basic} is the geometric form of the reduced Lagrangian mean curvature flow for the base curve.  We will mostly use the non-parametric form in this paper.
\begin{lem}[\cite{Lotay-Oliveira-1}*{Lemma~6.5}] \label{lem:LMCF-curve-graphical}
    Under Assumption~\ref{asmp:LMCF}, if \(\gm(\,\cdot\,,0)\) is given by the graph of a function \(u_0:[-1,1]\to\BR\) with \(\sup_{-1<x<1}|u'_0(x)|<\infty\),  then \(\gm(\,\cdot\,,t)\) remains graphical over \([-1,1]\).
    By writing \(\gm(\,\cdot\,,t)\) as \((x,u(x,t))\), the function \(u(x,t)\) satisfies
    \begin{align} \label{eqn:LMCF-graphical}
        \frac{\pl u}{\pl t} &= \frac{1}{\phi}\frac{u''}{1+(u')^2} ~,
    \end{align}
    with \(u(\pm1,t) = 0\), and
    \begin{align*}
        \sup_{-1<x<1}|u'(x,t)| \leq \sup_{-1<x<1}|u'_0(x)| ~.
    \end{align*}

    In this setting, if \(u(a_i,t)\neq0\) for \(i=2,\ldots,n-1\) and \(u''(x,t)\) remains bounded, then the flow exists.
\end{lem}
\begin{proof}
Since \(\sup_{-1<x<1}|u_0'(x)|<\infty\), we may choose a suitable branch of the Lagrangian angle so that \(-\frac{\pi}{2}<\bt<\frac{\pi}{2}\) on \(L_0\).  According to \cite{Smoczyk-99}*{Lemma~2.4}, the Lagrangian angle satisfies \(\frac{\pl}{\pl t}\bt = \Dt_{L_t}\bt\) along the flow.  By the maximum principle, \(\sup_{L_t}\bt \leq \sup_{L_0}\bt\) and \(\inf_{L_t}\bt \geq \inf_{L_0}\bt\).  Since \(\bt\) is the tangent angle of \(\gm(\,\cdot\,,t)\), the curve is always graphical over the \(x\)-axis, and \(\sup_{-1<x<1}|u'(x,t)| \leq \sup_{-1<x<1}|u'_0(x)|\).  It is a standard change-of-variable calculation to derive \eqref{eqn:LMCF-graphical}.

It remains to show that the boundedness of \(u''(x,t)\) and \(u(a_i,t)\neq0\) for \(i=2,\ldots,n-1\) imply the boundedness of
\begin{align} \label{eqn:fform-graphical}
        \frac{1}{\phi^{\frac{1}2}}\frac{u''}{\big(1+(u')^2\bigr)^{\frac{3}2}} ~,\quad
        \frac{1}{\phi^{\frac{1}2}}\dd\log\phi\Bigl(\frac{-u'\pl_x + \pl_y}{\sqrt{1+(u')^2}}\Bigr) ~,\quad \frac{1}{\phi^{\frac{1}2}}\dd\log\phi(\pl_z) ~.
\end{align}
For the first term in \eqref{eqn:fform-graphical}, its boundedness is immediate.  By \eqref{fct:phi-regular-part}, \(\phi^{-\frac{3}{2}}\dd\phi(\pl_z) = \phi^{-\frac{3}{2}}\pl_z\phireg\) on \(\Om\cap(\BR^2\times\{0\})\setminus\{(a_1,0),\cdots,(a_n,0)\}\), and is bounded.  Since the curves do not reach \((a_i,0)\) for \(i=2,\ldots,n-1\) and \(u'(x,t)\) is uniformly bounded, \({\phi^{-\frac{3}{2}}}\dd\phi\bigl({-u'\pl_x + \pl_y}\bigr)\) is bounded for \(-1+\dt\leq x\leq 1-\dt\) and at every \(t\), where \(\dt = \frac{1}{10}\min\{a_2-a_1,a_n-a_{n-1}\}\).  It remains to examine \({\phi^{-\frac{3}{2}}}\dd\phi\bigl({-u'\pl_x + \pl_y}\bigr)\) near \(a_1 = -1\) and \(a_n = 1\).  For any \(x\in(-1,-1+\dt)\), Taylor's theorem\footnote{In the mean-value form of the remainder, with \(x\) as the reference point rather than \(-1\).} says that \((x+1)u'(x,t)-u(x,t) = \frac{1}{2}(x+1)^2u''(\hat{x},t)\) for some \(\hat{x}\in(-1,x)\).  It follows that
\begin{align*}
    &\quad \bigl((x+1)^2+(u(x,t))^2\bigr)^{\frac{3}{4}}\cdot\Bigl|\bigl(-u'\pl_x + \pl_y\bigr)\bigl( \frac{1}{\sqrt{(x+1)^2+y^2}} \bigr) \Bigr|_{(x,y)=(x,u(x,t))} \\
    &= \frac{\frac{1}2(x+1)^2\bigl|u''(\hat{x},t)\bigr|}{\bigl((x+1)^2+(u(x,t))^2\bigr)^{\frac{3}4}} \leq \frac{1}{2} \bigl((x+1)^2+(u(x,t))^2\bigr)^{\frac{1}{4}} \bigl|u''(\hat{x},t)\bigr| ~.
\end{align*}
Therefore, the second term in \eqref{eqn:fform-graphical} is bounded if \(u''(x,t)\) is bounded.
\end{proof}

\subsection{The Maximum Principle}

We will assume that \(u_0\) is concave, and prove that the concavity is preserved along the flow (cf.\ \cite{Lotay-Oliveira-2}*{Proposition~3.4}).  We recall the standard maximum principle (cf.\ \cite{Hamilton-75}*{Part~IV}).
\begin{lem} \label{lem:max-principle}
    Let $f(x,t)$ be a continuous function on $[-1,1]\times[0,T)$ that is differentiable in $t$ on $(-1,1)\times(0,T)$.  Suppose that $f(-1,t)=f(1,t) \geq 0$ for all $t\in[0,T)$, and \(f(x,0)\geq0\) for all \(x\in[-1,1]\).  If there exists a constant \(c\in\BR\) such that
    \[ \frac{\pl f}{\pl t}(x_0,t_0) \geq c\,f(x_0,t_0) \]
    whenever $f$ attains a negative spatial minimum\footnote{More precisely, \(f(x_0,t_0)< 0\), and \(f(x_0,t_0)\leq f(x,t_0)\) for all \(x\in[-1,1]\).  Such an \(x_0\) will not be \(\pm1\).} at $(x_0,t_0)$, then $f\geq0$ on $[-1,1]\times[0,T)$.
\end{lem}

\begin{proof}
We first deal with a special case: suppose further that we have $\frac{\pl f}{\pl t}(x_0,t_0) > 0$ at every negative spatial minimum $(x_0,t_0)$ of $f$.

Suppose that \(f\) becomes negative somewhere.  Then, there exists a \(\dt>0\) such that the minimum of \(f\) on \([-1,1]\times[0,(1-\dt)T]\) is negative.  Let the minimum be achieved at \((x_0,t_0)\).  By assumption, one must have \((x_0,t_0)\in(-1,1)\times(0,T)\).  Since it attains a negative spatial minimum, \(\frac{\pl f}{\pl t}(x_0,t_0) > 0\).  On the other hand, \(f(x_0,t_0)\leq f(x_0,t)\) for all \(t\leq t_0\).  This is a contradiction.

Now, we deal with the general case.  Let \(\td{f} = e^{-(c+1)t}f\), and
\[ \frac{\pl\td{f}}{\pl t} = \bigl[-(c+1)f + \frac{\pl f}{\pl t}\bigr]e^{-(c+1)t} ~. \]
It is not hard to see that $\td{f}$ satisfies the assumptions of the special case.  This finishes the proof.
\end{proof}

To study the concavity of the solution to \eqref{eqn:LMCF-graphical}, consider
\begin{align} \label{fct:u-double-prime}
    P(x,t) &:= -\frac{\pl u}{\pl t} = -\frac{1}{\phi}\frac{u''}{1+(u')^2} ~.
\end{align}
According to Lemma~\ref{lem:LMCF-curve-graphical}, \(P(x,t)\) is continuous, and vanishes on \(\{\pm1\}\times[0,T)\).
By using \eqref{eqn:LMCF-graphical},
\begin{align}
    \frac{\pl P}{\pl t} 
    = \Bigl(\frac{\pl}{\pl y}\frac{1}{\phi}\Bigr)\cdot(-P)\cdot(\phi P) + \frac{1}{\phi}\frac{P''}{1+(u')^2} + \frac{2u'\cdot P\cdot P'}{1+(u')^2} ~. \label{eqn:u-tt}
\end{align}

\begin{lem} \label{lem:concavity}
    Under Assumption~\ref{asmp:LMCF}, if \(\gm(\,\cdot\,,0)\) is given by the graph of a \emph{concave} function \(u_0:[-1,1]\to\BR\) with \(\sup_{-1<x<1}|u'_0(x)|<\infty\),  then \(u(x,t)\) remains concave for all \(t\in[0,T)\), where \(T\) is the maximal existence time.  Moreover, \(u(x,t) > 0\) on \((-1,1)\times[0,T)\).
\end{lem}

\begin{rmk} \label{rmk:end-point-derivative}
    Under the concavity assumption, \(u'_0(\pm1)\) exists; see Lemma~\ref{lem:boundedness-kappa} and Remark~\ref{rmk:end-point-tangent}.  Also, \(u_0'(-1)=\sup_{-1<x<1}u'_0(x)\), and \(u_0'(1)=\inf_{-1<x<1}u'_0(x)\).
\end{rmk}

\begin{proof}
The concavity of \(u_0\) implies that \(P(x,0)\geq0\).
With the help of \eqref{fct:phi-regular-part}, we compute
\begin{align} \label{eqn:derivative-phi-inverse} \begin{split}
    \left. \frac{\pl}{\pl y}\frac{1}{\phi} \right|_{(x,y,0)} &= \Bigl(\frac{1}{2}\sum_{i=1}^n\frac{1}{\sqrt{(x-a_i)^2+y^2}} + \phireg(x,y,0)\Bigr)^{-2} \times \\
    &\qquad \Bigl( \frac{1}{2}\sum_{j=1}^n\frac{y}{\bigl((x-a_j)^2+y^2\bigr)^{\frac{3}2}} - \pl_y\phireg(x,y,0) \Bigr) ~.
\end{split} \end{align}
Since \(|y|\leq\sqrt{(x-a_i)^2+y^2}\), \(\pl_y\phi^{-1}|_{(x,y,0)}\) is bounded on \(\Om\cap(\BR^2\times\{0\})\setminus\{(a_1,0),\cdots,(a_n,0)\}\).  Note that
\begin{align*}
    \phi P &= -\frac{u''}{1+(u')^2} = -\sqrt{1+(u')^2}\cdot\kp ~.
\end{align*}
By Lemma~\ref{lem:boundedness-kappa} and Lemma~\ref{lem:LMCF-curve-graphical}, \(\phi P\) is bounded on\footnote{If \(T=\infty\), \(\phi\cdot P\) is bounded on any finite-time region.} \((-1,1)\times[0,T-\dt]\), for any sufficiently small \(\dt>0\).  It follows that Lemma~\ref{lem:max-principle} is applicable to \(P\) on \([-1,1]\times[0,T-\dt)\), and thus \(P\geq0\) on \([-1,1]\times[0,T-\dt)\).  One infers that the concavity is preserved along the flow.

If \(u(x,t) = 0\) for some \(x\in(-1,1)\) (at some \(t\)), concavity implies that \(u(x,t) = 0\) for all \(x\), which contradicts  \(u(a_2,t)\neq0\).  It follows that \(u(x,t) > 0\) on \((-1,1)\times[0,T)\).
\end{proof}

In the proof of the main theorem, we will use \(u(x,t)\) to bound \(u''(x,t)\).
\begin{cor} \label{cor:useful-equation}
    In the situation of Lemma~\ref{lem:concavity}, consider the function
    \begin{align} \label{fct:main-test}
        Q(x,t) := u(x,t) - g(t)\cdot P(x,t) = u + g(t)\frac{1}{\phi}\frac{u''}{1+(u')^2} ~,
    \end{align}
    where \(g(t)>0\) is a smooth function on \([0,\infty)\) to be chosen later.  It obeys
    \begin{align} \label{eqn:main-test}
        \frac{\pl Q}{\pl t} - \frac{1}{\phi}\frac{Q''}{1+(u')^2} &\geq
        \Bigl[ -\bigl(\pl_t g + c_0\frac{u}{\phi}+\frac{2(u')^2}{1+(u')^2}\bigr) + c_0\frac{Q}{\phi} + \frac{2u'Q'}{1+(u')^2} \Bigr]\cdot P
    \end{align}
    for some constant \(c_0 \geq 0\).
\end{cor}
\begin{proof}
In this setting, \(u\geq0\) and \(P\geq 0\).  For the first term on the right-hand side of \eqref{eqn:u-tt},
\begin{align*}
    \Bigl(\frac{\pl}{\pl y}\frac{1}{\phi}\Bigr)\cdot(-\phi P^2)
    &= \frac{P^2}{\phi} \Bigl( -\frac{1}{2}\sum_{i=1}^n\frac{u}{\bigl((x-a_i)^2+u^2\bigr)^{\frac{3}2}} + \pl_y\phireg(x,y,0)\bigl|_{y=u} \Bigr) \leq c_0\frac{P^2}{\phi} ~,
\end{align*}
where \(c_0 = \max_{\Om\cap(\BR^2\times\{0\})}|\pl_y\phireg(x,y,0)|\).  By \eqref{eqn:LMCF-graphical} and \eqref{eqn:u-tt},
\begin{align*}
    \frac{\pl Q}{\pl t} - \frac{1}{\phi}\frac{Q''}{1+(u')^2} &= -(\pl_tg)\cdot P - g\cdot \Bigl[ \Bigl(\frac{\pl}{\pl y}\frac{1}{\phi}\Bigr)\cdot(-\phi P^2) + \frac{2u'\cdot P\cdot P'}{1+(u')^2} \Bigr] \\
    &\geq \Bigl[-(\pl_tg) - g\cdot c_0\frac{P}{\phi} - g\cdot\frac{2u'\cdot P'}{1+(u')^2}\Bigr]\cdot P ~.
\end{align*}
Writing \(P\) as \(\frac{u-Q}{g}\) finishes the proof of this corollary.
\end{proof}

The maximum principle leads to the following avoidance principle.
\begin{lem} \label{lem:avoidance}
    Suppose that \(v_+(x,t),v_-(x,t)\in C^0([-1,1]\times[0,T))\cap C^\infty((-1,1)\times(0,T))\) satisfy the following.
    \begin{itemize}
        \item \(\frac{\pl}{\pl t}v_+ \geq \frac{1}{\phi}\frac{v''_+}{1+(v'_+)^2}\) and \(\frac{\pl}{\pl t} v_- \leq \frac{1}{\phi}\frac{v''_-}{1+(v_-')^2}\).
        \item The \(C^2\)-norms of \(v_+,v_-\) are bounded on \((-1,1)\times(0,T)\).
        \item \(v_\pm(\pm1,t) = 0\) for \(t\in[0,T)\), and \(v_\pm(a_i,t)\neq0\) for \(i=2,\ldots,n-1\).
        \item \(v_-(x,0) \leq v_+(x,0)\) for \(x\in[-1,1]\).
    \end{itemize}
    Then, \(v_- \leq v_+\) on \([-1,1]\times[0,T)\).
\end{lem}

\begin{proof}
We would like to apply Lemma~\ref{lem:max-principle} to \(v_+-v_-\) on \([-1,1]\times[0,T)\).
We compute
\begin{align*}
    \frac{\pl(v_+-v_-)}{\pl t} &\geq \frac{1}{\phi(x,v_+)}\frac{v''_+}{1+(v'_+)^2} - \frac{1}{\phi(x,v_-)}\frac{v''_-}{1+(v_-')^2} \\
    &= \frac{1}{\phi(x,v_+)}\frac{(v_+-v_-)''}{1+(v_+')^2} + \frac{v_-''}{1+(v_+')^2} \Bigl( \frac{1}{\phi(x,v_+)}-\frac{1}{\phi(x,v_-)} \Bigr) \\
    &\qquad\qquad + \frac{v_-''}{\phi(x,v_-)} \Bigl( \frac{1}{1+(v_+')^2}-\frac{1}{1+(v_-')^2} \Bigr) ~.
\end{align*}
By the assumptions of this lemma and the discussion after \eqref{eqn:derivative-phi-inverse}, the above expression at a negative spatial minimum of \(v_+-v_-\) is no less than \(c\cdot(v_+-v_-)\) for some \(c\in\BR\).  This completes the proof of the lemma.
\end{proof}

\section{The Barrier Function} \label{sec:barrier}

Let \(n\) be an integer greater than \(2\).  For \(-1 = a_1 < a_2 < \cdots < a_n = 1\) and \(\mass \geq 0\), let
\begin{align} \label{fct:phi-plane}
    \phixaxis(x,y) &= \mass + \frac{1}{2} \sum_{i=1}^n\frac{1}{\sqrt{(x-a_i)^2+y^2}} ~.
\end{align}
For most of this paper, the constant \(\mass\) plays no role, and may be set equal to zero.  {It is needed only in Corollary~\ref{cor:sharp-rate}, in particular for the ALF case.}  It contributes to lower-order terms in the construction of barrier functions.
In this section, we study the equation
\begin{align} \label{eqn:barrier}
    -\frac{\dd^2}{\dd x^2} w_\ld(x) &= \ld\cdot\phixaxis(x,w_\ld(x))\cdot w_\ld(x)
    \quad\text{for }-1<x<1 ~,
\end{align}
with \(w_\ld(\pm1) = 0\), for any \(\ld>0\).  A direct first-variation computation shows that solutions of \eqref{eqn:barrier} are critical points of the functional
\begin{align} \label{function:energy}
    \CE_\ld(f) &= \frac{1}{2}\int_{-1}^1 \Bigl|\frac{\dd f(x)}{\dd x}\Bigr|^2\dd x - \frac{\ld}{2}\int_{-1}^1\Bigl(\sum_{i=1}^n\sqrt{(x-a_i)^2+(f(x))^2} + \mass\cdot(f(x))^2\Bigr)\dd x
\end{align}
for \(f\in H^1_0((-1,1))\).

We summarize the properties of the solutions in the following theorem, the proof of which occupies this section.
\begin{thm} \label{thm:barrier-summary}
    For \(\ld\in(0,\frac{1}{\mass+1})\), there exists a \(w_\ld(x)\in C^\infty((-1,1))\cap C^1([-1,1])\) which satisfies \(\eqref{eqn:barrier}\) and has the following properties: \(w_\ld(x)\) is concave, \(w_\ld(\pm1)=0\), and \(0<w_\ld(x)\leq\frac{n\ld}{2}(1-x^2)\) for \(-1<x<1\).

    This one-parameter family of functions is differentiable in \(\ld\).  The derivative with respect to \(\ld\), \(\pl_\ld w_\ld(x)\), belongs to \(C^\infty((-1,1))\cap C^1([-1,1])\) and has the following properties: \(\pl_\ld w_\ld(x)\) is concave, \(\pl_\ld w_\ld(\pm1) = 0\), \(\pl_\ld w_\ld(x) \geq \frac{w_\ld(x)}{\ld} > 0\) for \(-1<x<1\).  It obeys the following equation
    \begin{align}
        -\frac{\dd^2}{\dd x^2}\pl_\ld w_\ld(x) &= \ld\cdot\msL_{\ld,0}(x)\cdot\pl_\ld w_\ld(x) + \phixaxis(x,w_\ld(x))\cdot w_\ld(x) ~,
    \end{align}
    where
    \begin{align}
        \msL_{\ld,0}(x) &= \mass + \frac{1}{2}\sum_{i=1}^n\frac{(x-a_i)^2}{\bigl((x-a_i)^2 + (w_\ld(x))^2\bigr)^{\frac{3}2}} < \phixaxis(x,w_\ld(x)) ~.
    \end{align}

    Moreover, there exists a constant \(c > 0\) such that
    \begin{align}
        \|w_\ld\|_{C^1([-1,1])} \leq c\,\|w_\ld\|_{C^0([-1,1])} \quad\text{and}\quad
        \|\pl_\ld w_\ld\|_{C^1([-1,1])} \leq c\,\|\pl_\ld w_\ld\|_{C^0([-1,1])} ~.
    \end{align}
\end{thm}

\subsection{Variational Formulation}

We first prove that \(\CE_\ld(f)\) always admits a nontrivial minimizer.

\begin{lem} \label{lem:existence-barrier}
    For any \(\ld>0\) with \(4\mass\ld<\pi^2\), there exists a \(w_\ld(x)\in C^\infty((-1,1))\cap C^0([-1,1])\) with the following properties.
    \begin{itemize}
        \item \(w_\ld(x) > 0\) for \(-1<x<1\), \(w_\ld(\pm1) = 0\).
        \item \(w_\ld(x)\) is concave.
        \item \(w_\ld(x)\) attains the infimum of \(\CE_\ld(f)\) for \(f\in H^1_0((-1,1))\), and satisfies \eqref{eqn:barrier}.
    \end{itemize}
    Moreover, \(\CE_\ld(0)\) is greater than the infimum of \(\CE_\ld(f)\).
\end{lem}
\begin{proof}
By the Cauchy--Schwarz inequality and the Poincar\'e inequality,
\begin{align*}
    \int_{-1}^1|f(x)|\dd x \leq \sqrt2 \bigl(\int_{-1}^1|f(x)|^2\dd x\bigr)^{\frac{1}{2}} \leq \frac{2\sqrt{2}}{\pi}\bigl(\int_{-1}^1|f'(x)|^2\dd x\bigr)^{\frac{1}{2}} ~.
\end{align*}
It follows that
\begin{align*}
    \CE_\ld(f) &\geq \frac{1}{2} \int_{-1}^1|f'(x)|^2\dd x - \frac{\ld}{2} \int_{-1}^1\sum_{i=1}^n(|x-a_i| + |f(x)|)\dd x {-\frac\ld2\int_{-1}^1\mass\cdot(f(x))^2} \\
    &\geq \frac{\pi^2-4\mass{\ld}}{2\pi^2} \int_{-1}^1|f'(x)|^2\dd x - \frac{\sqrt{2}\ld n}{\pi} \bigl(\int_{-1}^1|f'(x)|^2\dd x\bigr)^{\frac{1}{2}} - \frac{\ld}{2} \int_{-1}^1\sum_{i=1}^n|x-a_i|\dd x ~,
\end{align*}
which implies that \(\CE_\ld\) is coercive when \(4\mass\ld < \pi^2\); namely, \(\CE_\ld(f)\to\infty\) as \(\|f\|_{H^1_0}\to\infty\).

To show that \(\inf\{\CE_\ld(f):f\in H^1_0((-1,1))\} < \CE_\ld(0)\), choose a small bump function \(\chi\) supported on a neighborhood of a singular point other than \(\pm1\).  Specifically, fix \(0<c_1<\frac{1}{4}\min\{a_2-a_1,a_3-a_2\}\); let \(\chi(x) = 1\) when \(|x-a_2|\leq c_1\), and \(\chi(x) = 0\) when \(|x-a_2|\geq 2c_1\).  We compute for any \(\vep>0\),
\begin{align*}
    &\quad \int_{-1}^1\sum_{i=1}^n\bigl(\sqrt{(x-a_i)^2+(\vep\cdot\chi(x))^2} - |x-a_i|\bigr)\dd x \\
    &\geq 2\int_{0}^{c_1} (\sqrt{s^2 + \vep^2} - s)\dd s \\
    &= c_1\sqrt{c_1^2+\vep^2} - c_1^2 + \vep^2\log(c_1+\sqrt{c_1^2+\vep^2}) - \vep^2\log\vep
    > \vep^2 \log\frac{2c_1}{\vep} ~.
\end{align*}
Thus,
\begin{align*}
    \CE_\ld(0) - \CE_\ld(\vep\chi) &\geq - {\frac{\vep^2}2}\int_{-1}^1|\chi'(x)|^2\dd x + \frac{\ld}{2} \int_{-1}^1\sum_{i=1}^n\bigl(\sqrt{(x-a_i)^2+(\vep\cdot\chi(x))^2} - |x-a_i|\bigr)\dd x \\
    &> \vep^2\frac{\ld}{2} \Bigl( \log\frac{2c_1}{\vep} - \frac{c_2}{\ld} \Bigr)
\end{align*}
for some \(c_2>0\).  For \(\vep < 2c_1\exp({-\frac{c_2}{\ld}})\), \(\CE_\ld(0) > \CE_\ld(\vep\chi)\).

It is well known that the Dirichlet integral in \eqref{function:energy} is weakly lower semicontinuous.  By the Sobolev embedding theorem, \(H^1_0((-1,1))\hookrightarrow L^2((-1,1))\) is a compact embedding, and thus the potential term in \eqref{function:energy} is weakly continuous.  By coercivity and weak lower semicontinuity, the direct method in the calculus of variations shows that every minimizing sequence admits a weakly convergent subsequence whose limit attains the infimum of \(\CE_\ld(f)\).  Let \(w_\ld(x)\) be a minimizer of \(\CE_\ld(f)\).  By the Sobolev embedding theorem, \(w_\ld(x)\in C^0([-1,1])\) and \(w_\ld(\pm1) = 0\).

By \cite{Evans-10}*{\S5.10}, \(\CE_\ld(f) = \CE_\ld(|f|)\), and we may assume that \(w_\ld(x) \geq 0\).  Since \(\CE_\ld(0) > \CE_\ld(w_\ld(x))\), \(w_\ld(x)\not\equiv0\).  Since \(w_\ld(x)\geq0\) and satisfies \eqref{eqn:barrier} weakly, \(w_\ld(x)\) is concave.  Since \(w_\ld(x)\not\equiv0\), one infers that \(w_\ld(x) > 0\) for \(-1<x<1\).  Since \(w_\ld(x)\) satisfies \eqref{eqn:barrier}, \(w_\ld(x)\in C^\infty((-1,1))\).
\end{proof}

Lemma~\ref{lem:existence-barrier} does not assert the uniqueness of \(w_\ld(x)\), which will be established later.

\subsection{Estimates on the Solutions}
In this subsection, we derive \(C^0\) estimates on \(w_\ld(x)\).

\begin{lem} \label{lem:barrier-C0-estimate}
    For any \(\ld > 0\) with \(\mass\ld\leq1\), the solution \(w_\ld(x)\) obtained in Lemma~\ref{lem:existence-barrier} satisfies
    \begin{align} \label{eqn:barrier-C0-estimate}
        0 < w_\ld(x) \leq \frac{n\ld}{2}(1-x^2) \quad\text{for } -1<x<1
    \end{align}
    and
    \begin{align} \label{eqn:barrier-boundary-gradient-estimate} \begin{split}
        0 < w'_\ld(-1^+) := \lim_{x\to-1^+}\frac{w_\ld(x)}{x+1} = \lim_{x\to-1^+}w'_\ld(x) \leq {n\ld} ~, \\
        0 > w'_\ld(1^-):= \lim_{x\to1^-}\frac{w_\ld(x)}{x-1} = \lim_{x\to1^-}w'_\ld(x) \geq -{n\ld} ~.
    \end{split} \end{align}
\end{lem}
\begin{proof}
Since \(w_\ld(x)>0\) for \(-1<x<1\),
\begin{align*}
    0 < \phixaxis(x,w_\ld(x))\,w_\ld(x) < {\mass\cdot w_{\ld}(x)}+\frac{n}{2}
    \quad\text{and}\quad 0\leq \mass\cdot w_\ld(x) \leq \mass\cdot\|w_\ld\|_{C^0([-1,1])} ~.
\end{align*}
It follows that
\begin{align*}
    \frac{\dd^2}{\dd x^2}\bigl( w_\ld(x) - \frac{\ld}{4}(n+2\mass\cdot\|w_\ld\|_{C^0})(1-x^2) \bigr) &\geq \bigl(-\frac{\ld}{2}+\frac{\ld}{2}\bigr)(n+2\mass\cdot\|w_\ld\|_{C^0}) = 0 ~.
\end{align*}
The function \(w_\ld(x) - \frac{\ld}{4}(n+2\mass\cdot\|w_\ld\|_{C^0})(1-x^2)\) is convex, and vanishes on \(\pm1\).  Therefore, \(w_\ld(x) \leq \frac{\ld}{4}(n+2\mass\cdot\|w_\ld\|_{C^0})(1-x^2)\).  It follows that when \(\mass\ld \leq 1\), \(\|w_\ld\|_{C^0} \leq \frac{n\ld}{2}\), and
\begin{align*}
    \frac{w_\ld(x)}{1-x^2} &\leq \frac{n\ld}{2} ~.
\end{align*}

Since \(0 < - w''_\ld(x) < (\frac{n\ld}{2} + \ld\mass\|w_\ld\|_{C^0})\leq n\ld\), \(w'_\ld(x)\) is Lipschitz on \((-1,1)\) with Lipschitz constant no greater than \({n\ld}\).  It follows that \(\lim_{x\to 1^-}w'_\ld(x)\) exists, and is equal to \(\lim_{x\to 1^-}\frac{w_\ld(x)}{x-1}\).
Since $w_\ld$ is positive and concave on $(-1,1)$, \(w'_\ld(1^-) < 0\).  The lower bound for \(w'_\ld(1^-)\) follows from \eqref{eqn:barrier-C0-estimate}.  The argument for \(w'_\ld(-1^+)\) is completely the same.
\end{proof}

Note that \eqref{eqn:barrier-C0-estimate} immediately implies that \(w_\ld(x)\to0\) uniformly for \(-1\leq x\leq 1\) as \(\ld\to0\).  Since \(w'_\ld(-1^+) > 0\) and \(w'_\ld(1^-) < 0\),
\begin{align*}
    \inf\{\frac{w_\ld(x)}{1-x^2} : -1 < x < 1\} > 0 ~.
\end{align*}
The next task is to show that the above infimum has a positive lower bound for \(\ld\) bounded away from \(0\).  To this end, we need to show that
\begin{align} \label{function:inf-energy}
    \ul{\CE}(\ld) &= \inf\{\CE_\ld(f) : f\in H^1_0((-1,1))\}
\end{align}
is continuous.

\begin{lem} \label{lem:continuity-inf}
    \(\ul{\CE}(\ld)\) is concave, and hence continuous.
\end{lem}
\begin{proof}
    Since \( \CE_{t\td{\ld} + (1-t)\ld}(f) = t\,\CE_{\td{\ld}}(f) + (1-t)\,\CE_{\ld}(f) \) for \(0\leq t\leq 1\), taking the infimum over \(f\in H^1_0((-1,1))\) leads to the concavity of \(\ul{\CE}(\ld)\).
\end{proof}

\begin{lem} \label{lem:refined-C0-bound}
    For any \(\Ld_1 > \Ld_0 > 0\) with \(\mass\Ld_1\leq1\), there exists a constant \(c_3 = c_3(\Ld_0,\Ld_1) > 0\) such that the following holds.  For any \(\ld \in [\Ld_0,\Ld_1]\), any solution \(w_\ld(x)\) produced by Lemma~\ref{lem:existence-barrier} obeys
    \begin{align} \label{eqn:uniform-lower-bound}
        c_3(1-x^2) \leq w_\ld(x) \leq \frac{n\Ld_1}{2}(1-x^2) ~.
    \end{align}
\end{lem}
\begin{proof}
    The upper bound follows directly from \eqref{eqn:barrier-C0-estimate}.  We prove the lower bound by contradiction.  Suppose \eqref{eqn:uniform-lower-bound} fails.  Then, there exist
    \begin{itemize}
        \item a sequence \(\{\dt_k\}_{k\in\BN}\) with \(\dt_k > 0\) for all \(k\) and \(\dt_k \to0\) as \(k\to\infty\);\,
        \item a sequence of points \(\{x_k\}_{k\in\BN}\) with \(-1< x_k< 1\) for all \(k\),
        \item a sequence of parameters \(\{\ld_k\}_{k\in\BN}\) with \(\Ld_0\leq \ld_k\leq \Ld_1\) for all \(k\),
        \item a sequence of solutions \(\{w_{\ld_k}(x)\}_{k\in\BN}\) given by Lemma~\ref{lem:existence-barrier} for \(\ld_k\)
    \end{itemize}
    such that
    \begin{align} \label{eqn:lower-bound-contradiction}
        w_{\ld_k}(x_k) < \dt_k \bigl(1-(x_k)^2\bigr) \quad\text{for all } k~.
    \end{align}

    After passing to a subsequence, we may assume that \(\ld_k\to\hat\ld\) as \(k\to\infty\).  By \eqref{eqn:barrier-boundary-gradient-estimate} and \(0 > w''_{\ld_k}(x) \geq -n\Ld_1\), the Arzel\`a--Ascoli theorem ensures that (after passing to a subsequence,) \(\{w'_{\ld_k}(x)\}\) converges uniformly on \([-1,1]\).  Together with \eqref{eqn:barrier-C0-estimate}, the Arzel\`a--Ascoli theorem ensures that \(\{w_{\ld_k}(x)\}\) converges uniformly on \([-1,1]\).  To sum up, \(w_{\ld_k}(x)\) converges to some \(v_{\hat{\ld}}(x)\) in \(C^1([-1,1])\).  By the \(C^1\)-convergence and Lemma~\ref{lem:continuity-inf},
    \begin{align*}
        \CE_{\hat{\ld}}(v_{\hat\ld}) = \lim_{k\to\infty}\CE_{\ld_k}(w_{\ld_k}) = \lim_{k\to\infty}\ul{\CE}(\ld_k) = \ul{\CE}(\hat\ld) ~.
    \end{align*}
    It follows that \(v_{\hat\ld}(x)\) attains the infimum of \({\CE}_{\hat\ld}(~\cdot~)\).  Due to Lemma~\ref{lem:existence-barrier}, \(v_{\hat\ld}(x)\not\equiv0\) and satisfies \eqref{eqn:barrier} with parameter \(\hat\ld\).

    By construction, \(v_{\hat\ld}(x)\geq0\) and \(v_{\hat\ld}(\pm1) = 0\).  With the same argument as that of Lemma~\ref{lem:existence-barrier}, \(v_{\hat\ld}(x) > 0\) for \(-1<x<1\).  To utilize \eqref{eqn:lower-bound-contradiction}, we may assume that \(x_k\in[0,1)\) for all \(k\).  If \(x_k\in(-1,0]\) for all \(k\), the argument is completely parallel.  Since \(w_{\ld_k}(x)\) is concave,
    \begin{align*}
        &\quad (1-x_k)\,w_{\ld_k}(0) + x_k\,w_{\ld_k}(1) \leq w_{\ld_k}(x_k) < \dt_k \bigl(1-(x_k)^2\bigr) \\
        \Rightarrow &\quad w_{\ld_k}(0) \leq \dt_k(1+x_k) < 2\dt_k ~.
    \end{align*}
    Taking the limit as \(k\to\infty\), \(v_{\hat\ld}(0) \leq 0\), which contradicts \(v_{\hat\ld}(x) > 0\) for \(-1<x<1\).
\end{proof}

\subsection{Lipschitz Estimate in Parameter and Uniqueness}

The purpose of this subsection is to derive a Lipschitz estimate comparing two solutions obtained in Lemma~\ref{lem:existence-barrier}, which will be the key ingredient for the uniqueness and the differentiability in \(\ld\).

\begin{lem} \label{lem:key-Lipschitz-estimate}
    For any \(\Ld_1 > \Ld_0 > 0\) with \(\mass\Ld_1\leq 1\), there exists a constant \(c_4 = c_4(\Ld_0,\Ld_1) > 0\) such that the following holds.  For any \(\ld,\td\ld \in [\Ld_0,\Ld_1]\), let \(w(x)\) and \(\td{w}(x)\) be solutions produced by Lemma~\ref{lem:existence-barrier} with parameters \(\ld\) and \(\td{\ld}\), respectively.  Then,
    \begin{align} \label{eqn:key-inequality}
        \bigl| \td{w}(x) - w(x) \bigr| &\leq c_4\bigl| \td\ld - \ld \bigr|\min\{\td{w}(x),w(x)\}
    \end{align}
    for all \(x\in[-1,1]\).
\end{lem}

\begin{proof}
Consider their ratio and Wronskian:
\begin{align} \label{fct:ratio-and-Wronskian}
    \msR(x) = \frac{\td{w}(x)}{w(x)} \quad\text{and}\quad
    \msW(x) = w(x)\td{w}'(x) - w'(x)\td{w}(x) = (w(x))^2 \msR'(x) ~.
\end{align}
They are smooth functions on \((-1,1)\) and extend continuously to \(\pm1\) with
\begin{align*}
    \msR(-1) = \frac{\td{w}'(-1^+)}{w'(-1^+)} ~,\quad
    \msR(1) = \frac{\td{w}'(1^-)}{w'(1^-)} ~,\quad
    \msW(\pm1) = 0 ~.
\end{align*}
Since both \(w(x)\) and \(\td{w}(x)\) satisfy \eqref{eqn:barrier},
\begin{align} \label{eqn:Wronskian-derivative} \begin{split}
    \msW' &= w\td{w}'' - w''\td{w} =  \bigl(\ld\phixaxis(x,w) -\tilde\ld \phixaxis(x,\td{w})\bigr)w\td{w} \\
    \Rightarrow\quad (w^2\,\msR')' &= \ld\phixaxis(x,w)\Bigl(1 - \frac{\tilde\ld \phixaxis(x,\td{w})}{\ld\phixaxis(x,w)}\Bigr)w\td{w} ~.
\end{split} \end{align}

First, let \(x_0\in[-1,1]\) be a point where \(\msR(x)\) attains its maximum.  We claim that
\begin{align} \label{ineq:key1}
    \frac{\td\ld{ \phixaxis(x_0,\td{w}(x_0))}}{\ld{ \phixaxis(x_0,w(x_0))}}\geq 1~,
\end{align}
where \(\frac{{ \phixaxis(x,{\td w}(x))}}{{\phixaxis(x,w(x))}}\) extends continuously to \(\pm1\) with values
\[
    \frac{ \sqrt{1+w'(1^-)^2}}{ \sqrt{1+\td w'(1^-)^2}} \quad\text{and}\quad
    \frac{ \sqrt{1+w'(-1^+)^2}}{ \sqrt{1+\td w'(-1^+)^2}} ~,
\]
respectively.  We prove \eqref{ineq:key1} by contradiction; suppose \eqref{ineq:key1} fails.  If \(-1<x_0<1\), it follows from \eqref{eqn:Wronskian-derivative} that
\[ w(x_0)^2\msR''(x_0)+2w(x_0)w'(x_0)\msR'(x_0) > 0 ~, \]
which contradicts \(\msR'(x_0) = 0\) and \(\msR''(x_0)\leq 0\).  If \(x_0 = -1\), continuity ensures that there exists a \(\dt>0\) such that \(\frac{\tilde\ld \phixaxis(x,\td{w})}{\ld\phixaxis(x,w)} < 1\) for \(-1\leq x<-1+\dt\).  By \eqref{eqn:Wronskian-derivative},
\[ \bigl((w(x))^2\msR'(x)\bigr)' > 0 \quad\text{on }(-1,-1+\dt) ~. \]
Recall that \(\lim_{x\to-1^+} (w(x))^2\msR'(x) = \lim_{x\to-1^+}\msW(x) = 0\).  Hence, \(\msR(x) > 0\) on \((-1,-1+\dt)\), contrary to the assumption that \(\msR(x)\) attains its maximum at \(x_0 = -1\).  The case when \(x_0 = 1\) is similar.

To estimate \(\td{w}(x) - w(x)\), write
\begin{align}
    \frac{ \phixaxis(x,w(x))}{ \phixaxis(x,\td{w}(x))}-1
    &= \msS(x)\cdot(\td w(x)-w(x)) \label{ineq:key3}
\end{align}
where
\begin{align*}
    \msS(x) &= {\frac12}\frac{\td{w}(x)+w(x)}{\phixaxis(x,\td{w}(x))}\sum_{i=1}^n\Bigl[\frac{1}{\sqrt{(x-a_i)^2+(w(x)^2)}\sqrt{(x-a_i)^2+(\td{w}(x)^2)}}\times &\\
    &\qquad\qquad\qquad\qquad\qquad \frac{1}{\sqrt{(x-a_i)^2+(w(x)^2)}+\sqrt{(x-a_i)^2+(\td{w}(x)^2)}}\Bigr] ~.
\end{align*}
With the help of Lemma~\ref{lem:refined-C0-bound}, it is not hard to see that there exists a constant \(c_5 = c_5(\Ld_0,\Ld_1) > 0\) such that
\[ \frac{1}{c_5} \leq \msS(x)w(x) \leq c_5 ~. \]

It follows from \eqref{ineq:key3} that
\begin{align} \label{ineq:key2}
    \msR(x)-1 &= \frac{\td w(x)-w(x)}{w(x)} = \frac{1}{\msS(x)w(x)}\Bigl[\frac{ \phixaxis(x,w(x))}{ \phixaxis(x,\td{w}(x))}-1 \Bigr]  ~.
\end{align}
Now, fix \(x\in(-1,1)\), and assume that \(\msR(x) > 1 ~\Leftrightarrow~ \td{w}(x) > w(x)\).
By \eqref{ineq:key1} and \eqref{ineq:key2},
\begin{align*}
    \msR(x)-1\le \msR(x_0)-1\le c_5 \frac{\td\ld - \ld}{\ld} ~,
\end{align*}
and hence,
\begin{align*}
    \td{w}(x)-w(x) &\leq c_4(\td{\ld} - \ld)\cdot w(x)
\end{align*}
for some \(c_4 = c_4(\Ld_0,\Ld_1) > 0\).  When \(\msR(x) < 1\), one finds that \(w(x) - \td{w}(x) \leq c_4(\ld - \td{\ld})\cdot\td{w}(x)\) by switching \(w\) and \(\td{w}\).
\end{proof}

\begin{cor} \label{cor:unique-and-monotone}
    For any \(\ld \in(0,\frac{1}{1+\mass})\), the solution given by Lemma~\ref{lem:existence-barrier} with parameter \(\ld\) is unique.  Moreover, when \(\ld > \td{\ld} > 0\), \(w_\ld(x) > w_{\td{\ld}}(x)\) on \((-1,1)\).
\end{cor}

\begin{proof}
The uniqueness is a direct consequence of \eqref{eqn:key-inequality}.  When \(\ld > \td{\ld}\), it follows from \eqref{ineq:key1} and \eqref{ineq:key2} that \(\msR(x_0) < 1\).  Therefore, \(\msR(x) < 1\) on \([-1,1]\).
\end{proof}

\subsection{Differentiability in Parameter}

With the uniqueness of \(w_\ld(x)\), we can discuss its derivative with respect to \(\ld\).

\begin{lem} \label{lem:differentiable-in-lambda}
    The solution \(w_\ld(x)\) given by Lemma~\ref{lem:existence-barrier} is differentiable in \(\ld\in(0,\frac{1}{1+\mass})\).  Its derivative in \(\ld\), \(\pl_\ld w_\ld(x)\), belongs to \(C^1([-1,1])\), it is concave and positive on \((-1,1)\).  Moreover, \(\pl_x^\ell\pl_\ld w_\ld(x) = \pl_\ld\pl_x^\ell w_\ld(x)\) for \(-1< x < 1\) and \(\ell\in\BN\).
\end{lem}

\begin{proof}
For any \(\ld > 0\), choose \(\Ld_1 > \ld\) and \(\Ld_0 < \ld\), so that Lemma~\ref{lem:refined-C0-bound} and \ref{lem:key-Lipschitz-estimate} apply.  For any \(h\neq0\) with \(|h|\) sufficiently small, consider the difference quotient:
\[ q_{\ld,h}(x) = \frac{w_{\ld+h}(x) - w_\ld(x)}{h} ~. \]
To prove that \(\pl_\ld w_\ld(x)\) exists, it suffices to show that for any sequence \(\{h_k\}_{k\in\BN}\) with \(h_k\to 0\), there exists a subsequence \(\{h_{k_\ell}\}_{\ell\in\BN}\) such that \(\{q_{\ld,h_{k_\ell}}(x)\}\) converges, and the limiting function is unique.

To start, Lemma~\ref{lem:key-Lipschitz-estimate} implies that
\begin{align} \label{eqn:d-q-basic-estimate} \begin{split}
    &\bigl|q_{\ld,h}(x)\bigr|\leq c_4\,w_\ld(x) ~, \\
    &\bigl|q'_{\ld,h}(-1^+)\bigr|\leq c_4\,w_\ld'(-1^+)
    \quad\text{and}\quad
    \bigl|q'_{\ld,h}(1^-)\bigr|\leq -c_4\,w_\ld'(1^-) ~.
\end{split} \end{align} 

Since \(w_{\ld+h}(x)\) and \(w_\ld(x)\) are solutions to \eqref{eqn:barrier} (with parameters \(\ld+h\) and \(\ld\), respectively),
\begin{align} \label{eqn:difference-quotient}
    - q_{\ld,h}''(x) &= \ld\,\msL_{\ld,h}(x)\,q_{\ld,h}(x) + \msZ_{\ld,h}(x)  ~,
\end{align}
where
\begin{align} \label{fct:d-q-linear-coefficient} \begin{split}
    \msL_{\ld,h}(x) &= \mass + \frac{1}{2}\sum_{i=1}^n \Bigl[\frac{(x-a_i)^2\bigl(w_{\ld+h}(x)+w_\ld(x)\bigr)}{\sqrt{(x-a_i)^2+(w_{\ld+h}(x))^2}\sqrt{(x-a_i)^2+(w_\ld(x))^2}}\times \\
    &\qquad\qquad \frac{1}{w_\ld(x)\sqrt{(x-a_i)^2+(w_{\ld+h}(x))^2} + w_{\ld+h}(x)\sqrt{(x-a_i)^2+(w_\ld(x))^2}}\Bigr]
\end{split} \end{align}
and
\begin{align} \label{fct:d-q-zeroth-coefficient}
    \msZ_{\ld,h}(x) = \phixaxis(x,w_{\ld+h}(x))\,w_{\ld+h}(x) ~.
\end{align}
Note that the expressions of \(\msL_{\ld,h}\) and \(\msZ_{\ld,h}\) are both valid for \(h=0\).

By \eqref{eqn:uniform-lower-bound} and \eqref{eqn:d-q-basic-estimate}, there exists a constant \(c_6 = c_6(\ld) > 0\) such that \(\bigl|\msL_{\ld,h}(x)\,q_{\ld,h}(x)\bigr|\leq c_6\) on \((-1,1)\).  It is easy to see that \(0<\msZ_{\ld,h}(x)<{n}\) on \((-1,1)\).  By \eqref{eqn:difference-quotient}, \(|q''_{\ld,h}(x)| < c_6+{n}\).  With \eqref{eqn:d-q-basic-estimate} and a standard argument involving the Arzel\`a--Ascoli theorem, any sequence \(\{q_{\ld,h_k}(x) : h_k\to 0\}\) always admits a subsequence that converges in \(C^1\) on \([-1,1]\) and locally smoothly on \((-1,1)\), to a function \({q}_\ld(x)\) satisfying
\begin{align} \label{eqn:derivative-PDE}
    -{q}_\ld''(x) &= \ld\,\msL_{\ld,0}(x)\,{q}_\ld(x) + \msZ_{\ld,0}(x)
\end{align}
with \(|q_\ld(x)| \leq c_4\,w_\ld(x)\), \(|q'_\ld(-1^+)| \leq c_4\,w_\ld'(-1^+)\) and \(|q'_\ld(1^-)| \leq -c_4\,w'_\ld(1^-)\).

To show that \(q_\ld(x)\) is unique, it suffices to show that
\begin{align} \label{eqn:limit-difference}
    -v''(x) = \ld\,\msL_{\ld,0}(x)\,v(x)
\end{align}
with \(|v(x)|\leq 2c_4\,w_\ld(x)\), \(|v'(-1^+)| \leq 2c_4\,w_\ld'(-1^+)\) and \(|v'(1^-)| \leq -2c_4\,w_\ld'(1^-)\)
has only the trivial solution.  If not, we may assume that \(v(x)>0\) on some nontrivial interval $(x_0,x_1)$ with
\begin{align*}
    \lim_{x\to x_0^+}v(x) = \lim_{x\to x_1^-}v(x)=0 ~, \quad
    \lim_{x\to x_0^+}v'(x) \geq 0 \quad\text{and}\quad
    \lim_{x\to x_1^-}v'(x) \leq 0 ~.
\end{align*}
Note that
\begin{align*}
    \msL_{\ld,0}(x) &= \mass + \frac{1}{2}\sum_{i=1}^n \frac{(x-a_i)^2}{(x-a_i)^2+(w_\ld(x))^2}\frac{1}{\sqrt{(x-a_i)^2+(w_\ld(x))^2}} < \phixaxis(x,w_\ld(x))
\end{align*}
on \((-1,1)\).  Therefore,
\begin{align}
    -v''(x) &< \ld\,\phixaxis(x,w_\ld(x))\,v(x) \quad\text{on }(x_0,x_1) ~.
\end{align}
It follows that \(v''(x)w_\ld(x) - v(x)w''_\ld(x) > 0\) on \((x_0,x_1)\), and
\begin{align*}
    0 &< \int_{x_0}^{x_1} \bigl(v''(x)w_\ld(x) - v(x)w''_\ld(x)\bigr)\dd x \\
    &= \lim_{x\to x_1^-} \bigl(v'(x)w_\ld(x) - v(x)w'_\ld(x)\bigr) - \lim_{x\to x_0^+} \bigl(v'(x)w_\ld(x) - v(x)w'_\ld(x)\bigr) \\
    &= w_\ld(x_1)\lim_{x\to x_1^-} v'(x) - w_\ld(x_0) \lim_{x\to x_0^+} v'(x) \leq 0 ~,
\end{align*}
which is a contradiction.  Hence, \(q_\ld(x)\) is unique, and is the derivative of \(w_\ld(x)\) in \(\ld\).

By the same argument as that for Lemma~\ref{lem:barrier-C0-estimate}, \(\lim_{x\to-1^+}\frac{q_\ld(x)}{x+1}\) exists, and is equal to \(\lim_{x\to-1^+}q'_\ld(x)\).  A similar statement holds for \(x\to1^-\).  Hence, \(\pl_\ld w_\ld(x)\in C^1([-1,1])\).

Due to Corollary~\ref{cor:unique-and-monotone}, \(q_\ld(x) \geq 0\).  With \eqref{eqn:derivative-PDE}, \(q_\ld(x)\) is concave.  Since \(q_\ld(\pm1) = 0\) and \(0\) is not a solution of \eqref{eqn:derivative-PDE}, one must have \(q_\ld(x) > 0\) on \((-1,1)\).

The commutation of \(\partial_\lambda\) and \(\partial_x^\ell\) follows from standard interior regularity and difference quotient arguments.
\end{proof}

Since \(\msZ_{\ld,0}(x) = \phi(x,w_\ld(x))w_\ld(x)\), it follows from \eqref{eqn:derivative-PDE} and \eqref{eqn:barrier} that
\begin{align*}
    -\Bigl(\pl_\ld w_\ld(x) - \frac{w_\ld(x)}{\ld} \Bigr)'' &= \ld\,\msL_{\ld,0}(x)\,\pl_\ld w_\ld(x) \geq 0 ~.
\end{align*}
The function \(\pl_\ld w_\ld(x) - \frac{w_\ld(x)}{\ld}\) is concave, and vanishes at \(x=\pm1\).  It must be nonnegative on \([-1,1]\).

\begin{cor} \label{cor:derivative-and-quotient-in-lambda}
    When \(\ld\in(0,\frac{1}{1+\mass})\), the solution \(w_\ld(x)\) given by Lemma~\ref{lem:existence-barrier} satisfies
    \begin{align*}
        \pl_\ld w_\ld(x) \geq \frac{w_\ld(x)}{\ld} \quad\text{for } -1\leq x\leq 1 ~.
    \end{align*}
\end{cor}


\subsection{Bounding \(C^1\)-norm by \(C^0\)-norm}

To bound the derivative of \(w_\ld\), we first derive a simple calculus lemma.
\begin{lem} \label{lem:calculus-derivative-bound}
    For any \(c_8 > 0\) and \(\dt\in(0,\frac{1}{c_8})\), suppose that \(f(s)\) is a smooth function on $(0,\dt]$ with
\begin{align} \label{eqn:sample-condition}
    f(s) > 0 \quad,\quad 0\leq -f''(s) \leq \frac{2c_8 f(s)}{s} \quad\text{on } (0,\dt]\quad,
    \quad\text{and}\quad \lim_{s\to0^+} f(s) = 0 ~.
\end{align}
Then,
\begin{align} \label{eqn:sample-C1-by-C0}
    f'(0^+) \leq \frac{f(\dt)}{\dt(1-c_8\dt)} ~.
\end{align}
\end{lem}

\begin{proof}
We compute
\begin{align} \label{eqn:sample-ratio-derivative}
    \left(\frac{f(s)}{s(1-c_8s)}\right)' &= \frac{s(1-c_8s)\,f'(s) - (1-2c_8s)\,f(s)}{(s(1-c_8s))^2} ~.
\end{align}
By \eqref{eqn:sample-condition},
\begin{align*}
    \bigl(s(1-c_8s)\,f'(s) - (1-2c_8s)\,f(s)\bigr)' &= s(1-c_8s)\,f''(s) + 2c_8f(s) \geq 0 ~.
\end{align*}
Thus, \(s(1-c_8s)\,f'(s) - (1-2c_8s)\,f(s)\) has a limit \(L_0\in[-\infty,\infty)\) as \(s\to0^+\); indeed, \(L_0\) is its infimum over \((0,\dt)\).  By integrating \eqref{eqn:sample-ratio-derivative} from \(s_0\) to \(s_1\) for \(0<s_0<s_1<\!<\dt\),
\begin{align}
    \frac{f(s_1)}{s_1(1-c_8s_1)} - \frac{f(s_0)}{s_0(1-c_8s_0)} &\sim L_0\bigl( \frac{1}{s_0} - \frac{1}{s_1} \bigr) ~.
\end{align}
By fixing \(s_1\) and letting \(s_0\to0^+\), one finds that \(L_0 = \lim_{s_0\to0^+} f(s_0) = 0\).  It follows that \eqref{eqn:sample-ratio-derivative} is non-negative, and \(\frac{f(s)}{s(1-c_8s)}\) is non-decreasing.  With \(f(s)\geq0\), one infers that
\begin{align*}
    f'(0^+) &= \lim_{s\to0^+}\frac{f(s)}{s} = \lim_{s\to0^+}\frac{f(s)}{s(1-c_8s)} \quad\text{exists} ~,
\end{align*}
and
\begin{align*}
    f'(0^+) \leq \frac{f(\dt)}{\dt(1-c_8\dt)} ~.
\end{align*}
This finishes the proof of this lemma.
\end{proof}

\begin{lem} \label{lem:bound-C1-by-CO}
    There exists a constant \(c_9 > 0\) such that for any \(\ld\in(0,\frac{1}{1+\mass})\), the solution \(w_\ld(x)\) given by Lemma~\ref{lem:existence-barrier} satisfies
    \begin{align*}
        \|w_\ld\|_{C^1([-1,1])} &\leq c_9\|w_\ld\|_{C^0([-1,1])} \quad\text{and} \\
        \|\pl_\ld w_\ld\|_{C^1([-1,1])} &\leq c_9\|\pl_\ld w_\ld\|_{C^0([-1,1])} ~.
    \end{align*}
\end{lem}

\begin{proof}
By \eqref{eqn:barrier}, there exists \(c_{10} > 0\) and \(\dt\in(0,\frac{1}{2c_{10}})\) such that
\begin{align*}
    0\leq - w''_\ld(x) \leq \frac{2c_{10}w_\ld(x)}{x+1} \quad\text{on }(-1,-1+\dt] ~.
\end{align*}
According to Lemma~\ref{lem:calculus-derivative-bound}, \(w'_\ld(-1^+) \leq c_{11}\,w_\ld(-1+\dt)\).  By the same token, there is a similar upper bound for \(-w'_\ld(1^-)\).  We conclude from the concavity of \(w_\ld(x)\) that
\[ \|w_\ld\|_{C^1} \leq c_9\|w_\ld\|_{C^0} \]
for some constant \(c_9 > 0\).

For \(\pl_\ld w_\ld(x)\), it follows from \eqref{eqn:derivative-PDE} and Corollary~\ref{cor:derivative-and-quotient-in-lambda} that
\begin{align*}
    -\bigl(\pl_\ld w_\ld(x)\bigr)'' &\leq \ld\,\msL_{\ld,0}(x)\,\pl_\ld w_\ld(x) + \ld\phixaxis(x,w_\ld(x))\,\pl_\ld w_\ld(x) ~.
\end{align*}
With the same argument as above, the asserted estimate on \(\pl_\ld w_\ld(x)\) holds true.
\end{proof}


\section{Asymptotic Analysis of the Barrier} \label{sec:asymptotic}

The main purpose of this section is to analyze the asymptotic behavior of \(w_\ld\) given by Theorem~\ref{thm:barrier-summary}.  Since \(0\leq w_\ld(x)\leq\frac{n\ld}{2}(1-x^2)\), \(w_\ld\to 0\) as \(\ld\to0\) in \(C^0([-1,1])\).  For our purposes, we need to know the decay rate of \(\|w_\ld\|_{C^0}\) and the limit of the rescaled functions \(\frac{w_\ld(x)}{\|w_\ld\|_{C^0}}\) as \(\ld\to0\).

Not surprisingly, the configuration of the singular points plays an important role.  For \(-1 = a_1 < a_2 < \cdots < a_n = 1\), let \(A\in\Sym_{n-2}(\BR)\) be
\begin{align*}
    \begin{bmatrix}
        \frac{1}{a_2-a_1} + \frac{1}{a_3-a_2} & -\frac{1}{a_3-a_2} & 0   & \cdots & 0 \\ -\frac{1}{a_3-a_2} & \frac{1}{a_3-a_2} + \frac{1}{a_4-a_3} & -\frac{1}{a_4-a_3} & \ddots & \vdots \\ 0   & -\frac{1}{a_4-a_3} & \frac{1}{a_4-a_3} + \frac{1}{a_5-a_4} & \ddots & 0 \\ \vdots & \ddots & \ddots & \ddots & -\frac{1}{a_{n-1}-a_{n-2}} \\ 0 & \cdots & 0 & -\frac{1}{a_{n-1}-a_{n-2}} & \frac{1}{a_{n-1}-a_{n-2}} + \frac{1}{a_n - a_{n-1}}
    \end{bmatrix} ~.
\end{align*}

\begin{lem} \label{lem:leading-eigenvector}
    The matrix \(A\) has a unique eigenvector in \(\BR_{+}^{n-2}\), up to positive scalar multiple.
\end{lem}

\begin{proof}
For any \(c>0\), let \(\td{A}_c = c\,\bfI_{n-2} - A\).  When \(c\) is sufficiently large, every entry of \(\td{A}_c\) is nonnegative.  Such a matrix \(\td{A}_c\) is automatically irreducible, since its associated directed graph is strongly connected. The lemma follows directly from the Perron--Frobenius theorem.
\end{proof}

\begin{defn} \label{defn:eigenvalue-leading-linear-function}
    For the eigenvector given by Lemma~\ref{lem:leading-eigenvector}, denote by \(\mu\) the corresponding eigenvalue.  Let \(\mr{w}:[-1,1]\to[0,\infty)\) be the unique piecewise linear function characterized by the following properties.
    \begin{itemize}
        \item \(\mr{w}(\pm1) = 0\).
        \item \((\mr{w}(a_2),\cdots,\mr{w}(a_{n-1}))\) is the eigenvector given by Lemma~\ref{lem:leading-eigenvector} with \(\max_{1<i<n}\mr{w}(a_i) = 1\).
        \item \(\mr{w}\) is linear on each \([a_i,a_{i+1}]\) for \(i=1,\ldots,n-1\).
    \end{itemize}
\end{defn}

The eigenvalue-eigenvector equation reads
\begin{align}
    \frac{\mr{w}(a_i) - \mr{w}(a_{i-1})}{a_i - a_{i-1}} - \frac{\mr{w}(a_{i+1}) - \mr{w}(a_{i})}{a_{i+1} - a_{i}} &= \mu\,\mr{w}(a_i)
\end{align}
for \(i=2,\ldots,n-1\).  Summing over \(i\) gives
\begin{align} \label{formula.for.Ld}
    \mr{w}'(-1) - \mr{w}'(1) &= \mu\,\sum_{i=2}^{n-1} \mr{w}(a_i) = \mu\,\sum_{i=1}^{n} \mr{w}(a_i) ~.
\end{align}

The following lemma contains the calculation needed to determine the rescaled limit.
\begin{lem}\label{integration.lemma}
    For any positive \(\dt,\kp_0,\kp_1\), there exists a constant \(c = c(\dt,\kp_0,\kp_1) > 0\) with the following property.  Let \(\tau\in(0,1)\).  For any \(C^1\) function \(f:[0,\dt]\to(0,1)\) with
    \begin{align*}
        \frac{1}{\kp_0}\tau \leq f(s) \leq \kp_0\,\tau  \quad\text{and}\quad
        |f'(s)| \leq \kp_1\,\tau  \quad\text{for } s\in[0,\dt] ~,
    \end{align*}
    it satisfies
    \begin{align}
        \Bigl| \int_0^\dt \frac{\dd s}{\sqrt{s^2 + (f(s))^2}} + \log f(0) \Bigr| &\leq c ~, \label{integration.lemma1} \\
        \Bigl| \int_0^\dt \frac{\dd s}{({s^2 + (f(s))^2})^{\frac{3}2}} - \frac{1}{(f(0))^2} \Bigr| &\leq \frac{c}{\tau} ~, \label{integration.lemma3} \\
        \Bigl| \int_0^\dt \frac{s^2\,\dd s}{({s^2 + (f(s))^2})^{\frac{3}2}} + \log f(0) \Bigr| &\leq c ~. \label{integration.lemma2}
    \end{align}
\end{lem}

\begin{proof}
For \eqref{integration.lemma1},
\begin{align*}
    \int_0^\dt\frac{\dd s}{\sqrt{s^2 + (f(s))^2}}
    &= \int_0^\dt\frac{\dd s}{\sqrt{s^2 + (f(0))^2}}
    + \int_0^\dt \bigl(\frac{\dd s}{\sqrt{s^2 + (f(s))^2}} - \frac{\dd s}{\sqrt{s^2 + (f(0))^2}}\bigr)~.
\end{align*}
The first term on the right-hand side is \(\log(\dt + \sqrt{\dt^2+(f(0))^2}) - \log f(0)\), and note that \(\dt \leq \dt + \sqrt{\dt^2+(f(0))^2} \leq \dt + \sqrt{\dt^2+\kp_0^2}\).  For the second term on the right-hand side, it follows from the mean value theorem that
\begin{align*}
    \Bigl| \frac{1}{\sqrt{s^2 + (f(s))^2}} - \frac{1}{\sqrt{s^2 + (f(0))^2}} \Bigr|
    &\leq \frac{2\kp_0\kp_1\tau^2s}{(s^2+\kp_0^{-2}\tau^2)^{\frac{3}2}} ~,
\end{align*}
and therefore,
\begin{align*}
    \int_0^\dt \Bigl|\frac{1}{\sqrt{s^2 + (f(s))^2}} - \frac{1}{\sqrt{s^2 + (f(0))^2}}\Bigr|\dd s &\leq \int_0^\dt \frac{2\kp_0\kp_1\tau^2\,s\,\dd s}{(s^2+\kp_0^{-2}\tau^2)^{\frac{3}2}} < 2\kp_0^2\kp_1 ~.
\end{align*}
Putting these together proves \eqref{integration.lemma1}.

The arguments for \eqref{integration.lemma3} and \eqref{integration.lemma2} are similar, and are left as an exercise for the reader.
\end{proof}

\begin{prop} \label{limit.model}
The function \(\mr{w}(x)\) is the leading order part of \(w_\ld(x)\) in the following sense.  As \(\ld\to0\),
\begin{align*}
    \frac{w_\ld(x)}{\|w_\ld(\,\cdot\,)\|_{C^0([-1,1])}} \to \mr{w}(x) \quad\text{and}\quad
    \frac{\pl_\ld w_\ld(x)}{\|\pl_\ld w_\ld(\,\cdot\,)\|_{C^0([-1,1])}} \to \mr{w}(x)
\end{align*}
uniformly for \(-1\leq x\leq 1\).
Moreover, the above two convergences are \(C^1\) on any compact subset of \([-1,1]\setminus\{a_2,\cdots,a_{n-1}\}\).  It follows that
\begin{align*}
    \frac{w_\ld(x)}{\|w_\ld(\,\cdot\,)\|_{C^0([-1,1])}\cdot \mr{w}(x)} \to 1 \quad\text{and}\quad
    \frac{\pl_\ld w_\ld(x)}{\|\pl_\ld w_\ld(\,\cdot\,)\|_{C^0([-1,1])}\cdot \mr{w}(x)} \to 1
\end{align*}
uniformly for \(-1\leq x\leq 1\).
\end{prop}

\begin{proof}
{\textit{Step 1a: convergence of \(w_\ld\)}}.
Let
\begin{align*}
    \td{w}_\ld(x) = \frac{w_\ld(x)}{\|w_\ld(\,\cdot\,)\|_{C^0([-1,1])}} ~,
\end{align*}
which satisfies
\begin{align} \label{eqn:normalized-barrier}
    - \td{w}''_\ld(x) = \ld\phixaxis(x,w_\ld(x))\td{w}_\ld(x) ~.
\end{align}
Note that on any compact subset of \([-1,1]\setminus\{a_1,\cdots,a_n\}\), \(\phixaxis(x,w_\ld(x))\) is uniformly bounded.

Given a sequence \(\ld_j\to0\), we would like to show that there exists a subsequence \(\{\ld_{j_k}\}\) so that \(\td{w}_{\ld_{j_k}}(x) \to \mr{w}(x)\) in \(C^0([-1,1])\).  For brevity, we will omit the index of the sequence.  By construction, \(0\leq\td{w}_{\ld}\leq 1\); by Lemma~\ref{lem:bound-C1-by-CO}, \(\|\td{w}'_{\ld}\|_{C^0}\leq c\).  It follows from the Arzel\`a--Ascoli theorem that \(\td{w}_{\ld}(x)\to \xi(x)\) in \(C^0([-1,1])\).

To upgrade it to \(C^1\) convergence on any compact subset of \([-1,1]\setminus\{a_2,\cdots,a_{n-1}\}\), it suffices to show that \eqref{eqn:normalized-barrier} is uniformly bounded away from \(\{a_2,\cdots,a_{n-1}\}\).  This is obvious except near \(a_1 = -1\) and \(a_n = 1\).  Near \(a_1 = -1\), note that
\begin{align} \label{eqn:upgrade-near-pm1}
    \Bigl| \frac{1}{\sqrt{(x+1)^2+(w_\ld(x))^2}}\frac{w_\ld(x)}{\|w_\ld(\,\cdot\,)\|_{C^0}} \Bigr| &\leq \frac{1}{\|w_\ld(\,\cdot\,)\|_{C^0}}\cdot{\Bigl|\frac{w_\ld(x)}{x+1}\Bigr|} \leq \frac{\|w_\ld(\,\cdot\,)\|_{C^1}}{\|w_\ld(\,\cdot\,)\|_{C^0}} ~,
\end{align}
which is uniformly bounded by Lemma~\ref{lem:bound-C1-by-CO}.  It therefore remains to prove that $\xi = \mr{w}$.

{\textit{Step 1b: property of the limit}}.
Indeed, the above computation implies that \(\phixaxis(x,w_{\ld}(x))\,\td{w}_{\ld}(x)\) is uniformly bounded on any compact subset of \([-1,1]\setminus\{a_2,\cdots,a_{n-1}\}\), and thus \(\td{w}_{\ld}''\to0\) away from \(\{a_2,\cdots,a_{n-1}\}\).  By considering the limit of \((\td{w}'_{\ld}(x_2) - \td{w}'_{\ld}(x_1))(x_3-x_2) - (\td{w}'_{\ld}(x_3) - \td{w}'_{\ld}(x_2))(x_2-x_1)\), one finds that \(\xi\) is linear on \([a_i,a_{i+1}]\) for \(i=1,\ldots,n-1\).

Due to the \(C^0\) convergence, \(\xi\) is concave, \(\xi(-1) = 0 = \xi(1)\), and \(\max_{-1\leq x\leq 1} \xi(x) = 1\).  As a consequence, \(\xi(x) > 0\) for \(-1<x<1\), and \(\max_{2\leq i\leq n-1}\xi(a_i) = 1\).

{\textit{Step 1c: determining the limit}}.
When \(n=3\), this already implies that \(\xi = \mr{w}\).   When \(n>3\), fix any \(\dt>0\), and integrating \eqref{eqn:normalized-barrier} gives
\begin{align} \label{eqn:slope-limit-0}
    \frac{\td{w}'_{\ld}(a_i-\dt) - \td{w}'_{\ld}(a_i+\dt)}{\td{w}'_{\ld}(a_{i+1}-\dt) - \td{w}'_{\ld}(a_{i+1}+\dt)}
    &= \frac{\int_{a_i-\dt}^{a_{i}+\dt}\phixaxis(x,w_{\ld}(x))w_{\ld}(x)\dd x}{\int_{a_{i+1}-\dt}^{a_{i+1}+\dt}\phixaxis(x,w_{\ld}(x))w_{\ld}(x)\dd x}
\end{align}
for \(i=2,\ldots,n-2\).  As \(\ld\to0\), the left-hand side of \eqref{eqn:slope-limit-0} converges to
\begin{align} \label{eqn:slope-limit-0l}
    \frac{\xi'(a_i-\dt) - \xi'(a_i+\dt)}{\xi'(a_{i+1}-\dt) - \xi'(a_{i+1}+\dt)} &=
    \frac{\frac{\xi(a_i)-\xi(a_{i-1})}{a_i - a_{i-1}} - \frac{\xi(a_{i+1})-\xi(a_{i})}{a_{i+1} - a_{i}}}{\frac{\xi(a_{i+1})-\xi(a_{i})}{a_{i+1} - a_{i}} - \frac{\xi(a_{i+2})-\xi(a_{i+1})}{a_{i+2} - a_{i+1}}} ~.
\end{align}
For the right-hand side of \eqref{eqn:slope-limit-0}, choose \(\dt\) small enough so that \(\xi(a_i+s)\geq\frac{3}{4}\xi(a_i)\) whenever \(|s|\leq\dt\) and for \(i=2,\ldots,n-1\).  Since \(\td{w}_\ld\to \xi\) uniformly, for \(\ld<\!<1\),
\begin{align*}
    \frac{1}{2} \xi(a_i )\leq \td{w}_\ld(a_i+s) \leq 2\xi(a_i) \quad\text{whenever} |s|\leq\dt ~.
\end{align*}
Equivalently, we have
\begin{align} \label{eqn:slope-limit-sandwich}
    \frac{\xi(a_i)}{2}\cdot\|w_\ld\|_{C^0} \leq w_\ld(a_i+s) \leq 2\xi(a_i)\cdot\|w_\ld\|_{C^0} \quad\text{whenever} |s|\leq\dt ~.
\end{align}
We shall use the big \(O\) and little \(o\) notation for \(\ld\to0\).  According to \eqref{eqn:barrier-C0-estimate} and Lemma~\ref{lem:bound-C1-by-CO}, \(\|w_\ld\|_{C^0} = o(1)\), and
\begin{align*}
    |w_\ld'(x)| &\leq c\cdot\|w_\ld(\,\cdot\,)\|_{C^0} ~.
\end{align*}
By \eqref{integration.lemma1} of Lemma~\ref{integration.lemma} for \(\kp_0 = \frac{2}{\xi(a_i)}\), \(\kp_1 = c\) and \(\tau = \|w_\ld\|_{C^0}\),
\begin{align*}
    \int_{a_i-\dt}^{a_i+\dt} \frac{\dd x}{2\sqrt{(x-a_i)^2 + (w_\ld(x))^2}} &= - \log(w_\ld(a_i)) + O(1) \\
    &\stackrel{\eqref{eqn:slope-limit-sandwich}}{=} -\log\|w_\ld\|_{C^0} + O(1)
\end{align*}
for \(i=2,\ldots,n-1\).  After some simple manipulation,
\begin{align} \label{eqn:phi-w-asymptotic}
    \int_{a_i-\dt}^{a_i+\dt}\phixaxis(x,w_\ld(x))w_\ld(x)\dd x = -w_\ld(a_i)\log\|w_\ld\|_{C^0} + O(\|w_\ld\|_{C^0}) ~,
\end{align}
and thus the right hand side converges to \(\frac{\xi(a_i)}{\xi(a_{i+1})}\) as \(\ld\to0\).  To sum up, the limit of \eqref{eqn:slope-limit-0} as \(\ld\to0\) gives
\begin{align*}
    \frac{\frac{\xi(a_i)-\xi(a_{i-1})}{a_i - a_{i-1}} - \frac{\xi(a_{i+1})-\xi(a_{i})}{a_{i+1} - a_{i}}}{\frac{\xi(a_{i+1})-\xi(a_{i})}{a_{i+1} - a_{i}} - \frac{\xi(a_{i+2})-\xi(a_{i+1})}{a_{i+2} - a_{i+1}}} &= \frac{\xi(a_i)}{\xi(a_{i+1})}
\end{align*}
for \(i=2,\ldots,n-1\).  One infers that \(\bigl(\xi(a_2),\cdots,\xi(a_{n-1})\bigr)\) is an eigenvector of \(A\).  Due to Lemma~\ref{lem:leading-eigenvector}, Definition~\ref{defn:eigenvalue-leading-linear-function} and \(\max_{2\leq i\leq n-1}\xi(a_i) = 1\), \(\xi(a_i) = \mr{w}(a_i)\) for \(i=2,\ldots,n-1\), and thus \(\xi\equiv \mr{w}\).

{\textit{Step 2a: convergence of \(\pl_\ld w_\ld\)}}.
According to Lemma~\ref{lem:bound-C1-by-CO}, \(\frac{1}{\|\pl_\ld w_\ld(\,\cdot\,)\|_{C^0}}{|\pl_\ld w_\ld'(x)|} \leq c\) for \(-1\leq x\leq 1\).  The Arzel\`a--Ascoli theorem implies that
\begin{align} \label{eqn:convergence-barrier-derivative-parameter}
    \frac{1}{\|\pl_\ld w_\ld(\,\cdot\,)\|_{C^0}}{\pl_\ld w_\ld(x)} \to \eta(x)
    \quad\text{in } C^0([-1,1]) \text{ as }\ld\to0 ~.
\end{align}
It follows from the properties of \(\pl_\ld w_\ld\) and the \(C^0\)-convergence that \(\eta(-1) = 0 = \eta(1)\), \(\eta\) is concave and positive on \((-1,1)\), and \(\max_{-1\leq x\leq 1}\eta(x) = 1\) .

By \eqref{eqn:derivative-PDE}, Corollary~\ref{cor:derivative-and-quotient-in-lambda}, and the inequality \(\msL_{\ld,0}<\phixaxis(x,w_\ld(x))\),
\begin{align*}
    0\leq -\Bigl(\frac{\pl_\ld w_\ld}{\|\pl_\ld w_\ld\|_{C^0}}\Bigr)'' \leq 2\ld \phixaxis(x,w_\ld(x)) \frac{\pl_\ld w_\ld}{\|\pl_\ld w_\ld\|_{C^0}} ~.
\end{align*}
Since \(\phixaxis(x,w_\ld(x))\) is uniformly bounded on any compact subset of \([-1,1]\setminus\{a_1,\cdots,a_n\}\), the same argument as in Step 1a shows that the convergence \eqref{eqn:convergence-barrier-derivative-parameter} is \(C^1\) on any compact subset of \([-1,1]\setminus\{a_2,\cdots,a_{n-1}\}\) and \(\eta\) is linear on \([a_i,a_{i+1}]\) for \(i=0,\ldots,n-1\), and \(\max_{2\leq i\leq n-1}\eta(a_i) = 1\).

{\textit{Step 2b: asymptotics of \(C^0\) norms}}.  Before proving \(\eta \equiv \mr{w}\), we need to compare the asymptotic behavior of \(\|w_\ld\|_{C^0}\) and \(\|\pl_\ld w_\ld\|_{C^0}\).  By \eqref{eqn:barrier} and \eqref{eqn:derivative-PDE},
\begin{align*}
    0 &= \int_{-1}^{1} \bigl( (\pl_\ld w_\ld)'(w_\ld) - (\pl_\ld w_\ld)(w_\ld)' \bigl)' \dd x\\
    &= \int_{-1}^1 -\bigl( \ld(\msL_{\ld,0})(\pl_\ld w_\ld) + \phixaxis(x,w_\ld) w_\ld\bigr) (w_\ld) + (\pl_\ld w_\ld)\bigl(\ld\phixaxis(x,w_\ld)w_\ld\bigr) \dd x  ~.
\end{align*}
It follows that
\begin{align} \label{eqn:barrier-integration-C0}
    \int_{-1}^1\phixaxis(x,w_\ld)\cdot(w_\ld)^2 \dd x &= \ld\int_{-1}^1 \Bigl(\phixaxis(x,w_\ld) - \msL_{\ld,0}\Bigr) (w_\ld) (\pl_\ld w_\ld)\dd x ~.
\end{align}
We will analyze its asymptotic behavior as \(\ld\to 0\).

For the left-hand side of \eqref{eqn:barrier-integration-C0}, it follows from an argument similar to that for \eqref{eqn:phi-w-asymptotic} that
\begin{align*} 
    \int_{a_i-\dt}^{a_i+\dt}\phixaxis(x,w_\ld)\cdot(w_\ld)^2\dd x = -(w_\ld(a_i))^2\log\|w_\ld\|_{C^0} + O(\|w_\ld\|^2_{C^0})
\end{align*}
for \(i=2,\ldots,n-1\).  Since \(\phixaxis(x,w_\ld)\) is uniformly bounded away from \(a_1,\cdots,a_n\),
\begin{align*} 
    \int_{a_i+\dt}^{a_{i+1}-\dt}\phixaxis(x,w_\ld)\cdot(w_\ld)^2\dd x = O(\|w_\ld\|^2_{C^0})
\end{align*}
for \(i=1,\ldots,n-1\).  Near \(a_1 = -1\) (and \(a_n=1\)), it follows from Lemma~\ref{lem:bound-C1-by-CO} that
\begin{align*}
    \int_{-1}^{-1+\dt}\frac{(w_\ld(x))^2}{\sqrt{(x+1)^2+(w_\ld(x))^2}} \dd x \lesssim \|w_\ld\|^2_{C^0} \int_{-1}^{-1+\dt} (x+1)\dd x = O(\|w_\ld\|^2_{C^0}) ~.
\end{align*}
Hence,
\begin{align*}
    \int_{-1}^1\phixaxis(x,w_\ld)\cdot(w_\ld)^2 \dd x &=  -\log\|w_\ld\|_{C^0}\sum_{i=1}^n (w_\ld(a_i))^2 + O(\|w_\ld\|^2_{C^0}) ~.
\end{align*}
By using \(\frac{1}{\|w_\ld\|_{C^0}}w_\ld(a_i) = \mr{w}(a_i) + o(1)\), one finds that
\begin{align} \label{eqn:barrier-integration-C0-l}
    \lim_{\ld\to0} \frac{\int_{-1}^1\phixaxis(x,w_\ld)\cdot(w_\ld)^2 \dd x}{-\|w_\ld\|_{C^0}^2 \log\|w_\ld\|_{C^0}} &= \sum_{i=1}^n(\mr{w}(a_i))^2
\end{align}

For the right-hand side of \eqref{eqn:barrier-integration-C0}, we compute
\begin{align*}
    \phixaxis(x,w_\ld) - {\msL_{\ld,0}} &= \sum_{i=1}^n \frac{(w_\ld(x))^2}{2\bigl((x-a_i)^2 + (w_\ld(x))^2\bigr)^{\frac{3}2}} ~.
\end{align*}
It follows from \eqref{integration.lemma3} of Lemma~\ref{integration.lemma} that
\begin{align*}
    \int_{a_i-\dt}^{a_i+\dt} \sum_{j=1}^n \frac{1}{2\bigl((x-a_j)^2 + (w_\ld(x))^2\bigr)^{\frac{3}2}}\dd x = \frac{1}{(w_\ld(a_i))^2} + O(\frac{1}{\|w_\ld\|_{C^0}})
\end{align*}
for \(i=2,\ldots,n-1\).  Together with the mean value theorem and Lemma~\ref{lem:bound-C1-by-CO}, this implies that
\begin{align*}
    \int_{a_i-\dt}^{a_i+\dt} \Bigl(\phixaxis(x,w_\ld) - {\msL_{\ld,0}}\Bigr)(w_\ld)(\pl_\ld w_\ld)\dd x &= (w_\ld(a_i))(\pl_\ld w_\ld(a_i)) + O(\|w_\ld\|^2_{C^0} \|\pl_\ld w_\ld\|_{C^0})
\end{align*}
for \(i=2,\ldots,n-1\).  It is easy to see that
\begin{align*}
    \int_{a_i+\dt}^{a_{i+1}-\dt} \Bigl(\phixaxis(x,w_\ld) - {\msL_{\ld,0}}\Bigr)(w_\ld)(\pl_\ld w_\ld)\dd x &= O(\|w_\ld\|^3_{C^0} \|\pl_\ld w_\ld\|_{C^0})
\end{align*}
for \(i=1,\ldots,n-1\).  Near \(a_1 = -1\) (and \(a_n=1\)), it follows from Lemma~\ref{lem:bound-C1-by-CO} that
\begin{align*}
    &\quad \int_{-1}^{-1+\dt}\frac{(w_\ld(x))^2}{\bigl((x+1)^2+(w_\ld(x))^2\bigr)^{\frac{3}2}}(w_\ld(x))(\pl_\ld w_\ld(x)) \dd x \\
    &\lesssim \|w_\ld\|^3_{C^0}\|\pl_\ld w_\ld\|_{C^0} \int_{-1}^{-1+\dt} (x+1)\dd x = O(\|w_\ld\|^3_{C^0} \|\pl_\ld w_\ld\|_{C^0}) ~.
\end{align*}
Thus,
\begin{align*}
    \int_{-1}^1 \Bigl(\phixaxis(x,w_\ld) - \msL_{\ld,0}\Bigr) (w_\ld)(\pl_\ld w_\ld)\dd x
    &= \sum_{i=1}^n(w_\ld(a_i))(\pl_\ld w_\ld(a_i)) + O(\|w_\ld\|^2_{C^0} \|\pl_\ld w_\ld\|_{C^0}) ~.
\end{align*}
By using \(\frac{1}{\|w_\ld\|_{C^0}}w_\ld(a_i) = \mr{w}(a_i) + o(1)\) and \(\frac{1}{\|\pl_\ld w_\ld\|_{C^0}}\pl_\ld w_\ld(a_i) = \eta(a_i) + o(1)\),
\begin{align} \label{eqn:barrier-integration-C0-r}
    \lim_{\ld\to0} \frac{\int_{-1}^1 \Bigl(\phixaxis(x,w_\ld) - \msL_{\ld,0}\Bigr) (w_\ld)(\pl_\ld w_\ld)\dd x}{\|w_\ld\|_{C^0} \|\pl_\ld w_\ld\|_{C^0}} &= \sum_{i=1}^n\mr{w}(a_i)\eta(a_i) ~.
\end{align}

Putting \eqref{eqn:barrier-integration-C0}, \eqref{eqn:barrier-integration-C0-l} and \eqref{eqn:barrier-integration-C0-r} together gives
\begin{align} \label{eqn:asymptotic-C0}
    \lim_{\ld\to0} \frac{\|w_\ld\|_{C^0} \cdot (-\log\|w_\ld\|_{C^0})}{\ld\cdot\|\pl_\ld w_\ld\|_{C^0}} &= \frac{\sum_{i=1}^n\mr{w}(a_i)\eta(a_i)}{\sum_{i=1}^n(\mr{w}(a_i))^2} ~.
\end{align}

{\textit{Step 2c: determining the limit}}.
It remains to show that \(\eta \equiv \mr{w}\) when \(n>3\).  By the same argument as in step 1c,
\begin{align} \label{eqn:eta-slope-limit}
    \frac{\frac{\eta(a_i)-\eta(a_{i-1})}{a_i - a_{i-1}} - \frac{\eta(a_{i+1})-\eta(a_{i})}{a_{i+1} - a_{i}}}{\frac{\eta(a_{i+1})-\eta(a_{i})}{a_{i+1} - a_{i}} - \frac{\eta(a_{i+2})-\eta(a_{i+1})}{a_{i+2} - a_{i+1}}}
    &= \lim_{\ld \to 0} \frac{\int_{a_i-\dt}^{a_i+\dt}\bigl(\ld(\msL_{\ld,0})(\pl_\ld w_\ld) + \phixaxis(x,w_\ld) w_\ld\bigr) \dd x}{\int_{a_{i+1}-\dt}^{a_{i+1}+\dt}\bigl(\ld(\msL_{\ld,0})(\pl_\ld w_\ld) + \phixaxis(x,w_\ld) w_\ld\bigr) \dd x}
\end{align}
for \(i=2,\ldots,n-2\).

It follows from \eqref{integration.lemma2} of Lemma~\ref{integration.lemma} that
\begin{align*}
    \int_{a_i-\dt}^{a_i+\dt}\ld\msL_{\ld,0}\,\dd x &= \ld\Bigl( \int_{a_i-\dt}^{a_i+\dt}\frac{(x-a_i)^2}{2\bigl({(x-a_i)^2+(w_\ld)^2\bigr)^{\frac{3}2}}} \dd x + O(1) \Bigr) \\
    &= \ld\bigl( -\log w_\ld(a_i) + O(1) \bigr) \\
    &\stackrel{\eqref{eqn:slope-limit-sandwich}}{=} \ld\bigl( -\log\|w_\ld\|_{C^0} + O(1) \bigr) ~.
\end{align*}
With the help of Lemma~\ref{lem:bound-C1-by-CO}, it implies that
\begin{align*}
    \int_{a_i-\dt}^{a_i+\dt}\ld(\msL_{\ld,0})(\pl_\ld w_\ld)\,\dd x &= \ld\bigl( -(\pl_\ld w_\ld(a_i))\cdot\log\|w_\ld\|_{C^0} + O(\|\pl_\ld w_\ld\|_{C^0}) \bigr) ~.
\end{align*}
Together with \eqref{eqn:phi-w-asymptotic},
\begin{align*}
    &\quad {\int_{a_i-\dt}^{a_i+\dt}\bigl(\ld(\msL_{\ld,0})(\pl_\ld w_\ld) + \phixaxis(x,w_\ld) w_\ld\bigr) \dd x} \\
    &= (-\log\|w_\ld\|_{C^0}) \bigl( \ld\cdot\pl_\ld w_\ld(a_i) + w_\ld(a_i) \bigr) + O(\ld\cdot\|\pl_\ld w_\ld\|_{C^0}) ~.
\end{align*}
It follows that
\begin{align*}
    &\quad\lim_{\ld\to0}\frac{{\int_{a_i-\dt}^{a_i+\dt}\bigl(\ld(\msL_{\ld,0})(\pl_\ld w_\ld) + \phixaxis(x,w_\ld) w_\ld\bigr) \dd x}}{\ld\cdot(-\log\|w_\ld\|_{C^0})\cdot\|\pl_\ld w_\ld\|_{C^0}} \\
    &= \lim_{\ld\to 0} \Bigl( \frac{\pl_\ld w_\ld(a_i)}{\|\pl_\ld w_\ld\|_{C^0}} + \frac{w_\ld(a_i)}{\ld\cdot\|\pl_\ld w_\ld\|_{C^0}} \Bigr) \\
    &= \eta(a_i) + \lim_{\ld\to 0}\frac{w_\ld(a_i)}{\|w_\ld\|_{C^0}} \frac{\|w_\ld\|_{C^0} \cdot (-\log\|w_\ld\|_{C^0})}{\ld\cdot\|\pl_\ld w_\ld\|_{C^0}}\frac{1}{(-\log\|w_\ld\|_{C^0})} ~,
\end{align*}
and the last limit vanishes by \eqref{eqn:asymptotic-C0} and \(\|w_\ld\|_{C^0}\to0\).  By the same argument as in step 1c, \(\eta\equiv\mr{w}\).
\end{proof}

After showing \(\eta\equiv\mr{w}\), the limit \eqref{eqn:eta-slope-limit} is equal to \(1\).  The following lemma is a refined version of this limit.
\begin{lem} \label{lem:C0-asymptotic}
The \(C^0\)-norm of \(w_\ld(x)\) satisfies
\begin{align*}
    \lim_{\ld\to 0}\ld\cdot(-\log\|w_\ld(\,\cdot\,)\|_{C^0([-1,1])}) = \mu ~,
\end{align*}
the positive eigenvalue given by Lemma~\ref{lem:leading-eigenvector} (see Definition~\ref{defn:eigenvalue-leading-linear-function}).
\end{lem}

\begin{proof}
It follows from \eqref{eqn:barrier} that
\begin{align*}
    \frac{w_\ld'(-1) - w_\ld'(1)}{\ld} &= \int_{-1}^1\phixaxis(x,w_\ld)w_\ld\,\dd x ~.
\end{align*}
By \eqref{eqn:phi-w-asymptotic} and a direct argument near \(\pm1\),
\begin{align*}
    \int_{-1}^1\phixaxis(x,w_\ld)w_\ld\,\dd x &= (-\log\|w_\ld\|_{C^0})\sum_{i=1}^n w_\ld(a_i) + O(\|w_\ld\|_{C^0}) ~.
\end{align*}
According to \eqref{formula.for.Ld} and Proposition~\ref{limit.model},
\begin{align*}
    \mu\,\sum_{i=1}^n\mr{w}(a_i) &= \mr{w}'(-1) - \mr{w}'(1) \\
    &= \lim_{\ld\to 0} \frac{w_\ld'(-1) - w_\ld'(1)}{\|w_\ld\|_{C^0}} \\
    &= \lim_{\ld\to 0} \Bigl( -(\ld\cdot\log\|w_\ld(\,\cdot\,)\|_{C^0([-1,1])})\frac{\sum_{i=1}^n w_\ld(a_i)}{\|w_\ld\|_{C^0}} + O(\ld)\Bigr) ~,
\end{align*}
and this lemma follows.
\end{proof}

As a consequence of this lemma and \eqref{eqn:asymptotic-C0},
\begin{align} \label{eqn:asymptotic-C0-refined}
    \lim_{\ld\to0} \frac{\mu}{\ld^2}\cdot\frac{\|w_\ld\|_{C^0}}{\|\pl_\ld w_\ld\|_{C^0}} &= 1 ~.
\end{align}

\section{Proof of the Main Result} \label{sec:main-argument}

\begin{prop} \label{prop:sup-sub-solution}
    Fix any \(\mass\geq0\).  Consider the functions \(\{w_\ld(x)\}_{0<\ld<\frac{1}{\mass+1}}\) given by Theorem~\ref{thm:barrier-summary}, and let \(\mu>0\) be the constant introduced in Definition~\ref{defn:eigenvalue-leading-linear-function}.  Let
    \begin{align*}
        \br{c} = \sup_{\Om\cap(\BR^2\times\{0\})}\frac{\phi(x,y,0)}{\phixaxis(x,y)} \quad\text{and}\quad \ul{c} = \inf_{\Om\cap(\BR^2\times\{0\})}\frac{\phi(x,y,0)}{\phixaxis(x,y)} ~.
    \end{align*}

    For any \(\ep\in(0,1)\), let
    \begin{align}
        \br{\ld}_{\ep}(t) = \sqrt{\frac{\br{c}\cdot\mu}{2(1-\ep)t}} \quad\text{and}\quad \ul{\ld}_{\ep}(t) = \sqrt{\frac{\ul{c}\cdot\mu}{2(1+\ep)t}} ~.
    \end{align}
    Then, there exists \(t_0>0\) such that \(\bar{w}(x,t;\ep) := w_{\br{\ld}_{\ep}(t)}(x)\) is a supersolution to \eqref{eqn:LMCF-graphical} and \(\ul{w}(x,t;\ep) := w_{\ul{\ld}_{\ep}(t)}(x)\) is a subsolution to \eqref{eqn:LMCF-graphical} whenever \(t\geq t_0\).
\end{prop}

\begin{proof}
Consider \(w_{\ld(t)}(x)\) for any smooth \(\ld(t)>0\) with \(\ld(t)\to 0\) as \(t\to\infty\).  By \eqref{eqn:barrier},
\begin{align*}
    &\quad \frac{\pl}{\pl t}w_{\ld(t)}(x) - \frac{1}{\phi}\frac{w_{{\ld}(t)}''}{1+\big(w_{\ld(t)}'\big)^2} \\
    &= \pl_\ld w_\ld\bigr|_{\ld=\ld(t)}\cdot\pl_t\ld(t) + \frac{\phixaxis}{\phi}\frac{\ld(t)\cdot w_{\ld(t)}}{1+\big(w_{\ld(t)}'\big)^2} \\
    &= (\pl_\ld w_\ld)\cdot \biggl[ \pl_t\ld + \frac{\phixaxis}{\phi}\frac{1}{1+\big(w_{\ld}'\big)^2} \frac{w_{\ld}}{\|w_\ld\|_{C^0}\cdot\mr{w}}\cdot\frac{\|\pl_\ld w_\ld\|_{C^0}\cdot\mr{w}}{(\pl_\ld w_\ld)} \frac{\mu\|w_\ld\|_{C^0}}{\ld^2\|\pl_\ld w_\ld\|_{C^0}}\cdot\Bigl(\frac{\ld^3}{\mu}\Bigr) \biggr] ~.
\end{align*}
According to Proposition~\ref{limit.model}, \eqref{eqn:asymptotic-C0-refined} and Theorem \ref{thm:barrier-summary},
\[ \frac{1}{1+\big(w_{\ld}'\big)^2} \frac{w_{\ld}}{\|w_\ld\|_{C^0}\cdot\mr{w}}\cdot\frac{\|\pl_\ld w_\ld\|_{C^0}\cdot\mr{w}}{(\pl_\ld w_\ld)} \frac{\mu\|w_\ld\|_{C^0}}{\ld^2\|\pl_\ld w_\ld\|_{C^0}} \stackrel{\ld\to0}{\longrightarrow} 1 \]
uniformly for \(-1\leq x\leq 1\).  Since \(\bar{c}^{-1}\leq\frac{\phixaxis}{\phi}\leq\ul{c}^{-1}\),
\[ \pl_t \br{\ld}_{\ep}(t) = -\frac{1-\ep}{\bar{c}\mu}\bigl(\br{\ld}_{\ep}(t)\bigr)^3 \quad\text{and}\quad \pl_t \ul{\ld}_{\ep}(t) = -\frac{1+\ep}{\ul{c}\mu}\bigl(\ul{\ld}_{\ep}(t)\bigr)^3 ~. \]
This completes the proof of the proposition.
\end{proof}

\begin{thm} \label{thm:main}
    For any $\Lambda>0$, there exists a constant \(c > 0\) with the following property.  Suppose that \(L_0\) is a circle-invariant Lagrangian \(2\)-sphere passing through \(p_1\) and \(p_n\) such that \(\pr(L_0) = \{(x,u_0(x),0):-1\leq x\leq 1\}\) where \(u_0\) is concave and satisfies 
    \begin{align} \label{eqn:main-condition}
        |u_0'(x)|\leq \frac{1}{c} \qquad\text{and}\qquad {-u_0''(x)} \leq \Lambda\cdot u_0(x)\cdot \phi(x,u_0(x),0) ~.
    \end{align}
    Then, the Lagrangian mean curvature flow \(\{L_t\}\) exists for all time, and 
    \begin{align} \label{eqn:condition-later-time}
        {-u''(x,t)} \leq 4\Lambda\cdot u(x,t)\cdot \phi(x,u(x,t),0)
    \end{align}
    for all $t\geq0$.
    Moreover, the mean curvature converges to \(0\) uniformly as \(t\to\infty\) while
    \begin{align} \label{eqn:blow-up-rate}
        \sqrt{\frac{\mu}{2\bar{c}}} \leq \liminf_{t\to\infty}\frac{\log\max|A(\,\cdot\,,t)|}{\sqrt{t}} \leq \limsup_{t\to\infty}\frac{\log\max|A(\,\cdot\,,t)|}{\sqrt{t}} \leq \sqrt{\frac{\mu}{2\ul{c}}} ~;
    \end{align}
    see Proposition~\ref{prop:sup-sub-solution} for the constants \(\bar{c}\), \(\ul{c}\), \(\mu\). 
\end{thm}

\begin{proof}
{\textit{Step 1: choosing \(c_1\) and \(g(t)\)}}.
Let \(c_0\) be the constant in Corollary~\ref{cor:useful-equation}.  For any \(\ep\in(0,\frac12]\), consider the supersolution \(\bar{w}(x,t;\ep)\) and subsolution \(\ul{w}(x,t;\ep)\) to \eqref{eqn:LMCF-graphical} for \(t\geq t_0(\ep)\), produced by Proposition~\ref{prop:sup-sub-solution}.
By enlarging $t_0=t_0(\Ld,\ep)>0$, we may assume
\begin{align} \label{condition:constant-2}
    \frac{1}{2\Ld}  > \Bigl( c_0\sup_{\Om\cap(\BR^2\times\{0\})}\frac{1}{\phi} + 2 \Bigr)\cdot \int_{0}^\infty \bigl[ \|\bar{w}(\,\cdot\,,\tau+t_0;\ep)\|_{C^0([-1,1])} + \|\bar{w}(\,\cdot\,,\tau+t_0;\ep)\|^2_{C^1([-1,1])} \bigr]\dd\tau ~,
\end{align}
and \(\|\bar{w}(\,\cdot\,,t+t_0;\ep)\|_{C^1([-1,1])} \leq 1\) for every \(t\geq0\).

Then, let
\begin{align}
    c_1(\Ld,\ep) &= \sup_{-1<x<1}\frac{\min\{1-x,1+x\}}{\bar{w}(x,t_0;\ep)} ~.
\end{align}
Note that
\begin{align} \label{condition:constant-1}
    \min\{1-x,1+x\}\leq c_1(\Ld,\ep)\cdot\bar{w}(x,t_0;\ep) \quad\text{for every } x\in[-1,1] ~.
\end{align}
The existence of such \(t_0\) and $c_1$ is guaranteed by Lemma~\ref{lem:refined-C0-bound}, Theorem~\ref{thm:barrier-summary} and Lemma~\ref{lem:C0-asymptotic}.

Let
\begin{align} \label{fct:test-coefficient}
    g(t) &= \frac1\Ld - \Bigl( c_0\sup_{\Om\cap(\BR^2\times\{0\})}\frac{1}{\phi} + 2 \Bigr)\cdot \int_{0}^t \bigl[ \|\bar{w}(\,\cdot\,,\tau+t_0;\ep)\|_{C^0([-1,1])} + \|\bar{w}(\,\cdot\,,\tau+t_0;\ep)\|^2_{C^1([-1,1])} \bigr]\dd\tau ~.
\end{align}

{\textit{Step 2: long-time existence}}.  Now, let \(c = c_1(\Ld,\frac12)\).  In this step and the next step, write \(\bar{w}(x,t;\frac12)\) as \(\bar{w}(x,t)\), and \(\ul{w}(x,t;\frac12)\) as \(\ul{w}(x,t)\) for brevity.  Denote by \(T\in(0,\infty]\) the maximal existence time of the Lagrangian mean curvature flow.  Since \(u_0\) is non-trivial and concave, \eqref{eqn:barrier-C0-estimate} implies that there exists a \(t_1>t_0\) such that \(u_0(x)\geq\ul{w}(x,t_1)\) for \(x\in[-1,1]\).  By Lemma~\ref{lem:avoidance}, \(u(x,t)\geq\ul{w}(x,t+t_1)\) for every \(t\in[0,T)\).  In particular, \(u(a_i,t)>0\) for \(t\in[0,T)\) and \(i=2,\ldots,n-1\).  In other words, the curve does not reach \((a_2,0),\ldots,(a_{n-1},0)\) in finite time.

The conditions \eqref{eqn:main-condition}, \eqref{condition:constant-1} and concavity imply that \(u_0(x) \leq \bar{w}(x,t_0)\) for \(x\in[-1,1]\).  Again by Lemma~\ref{lem:avoidance}, \(u(x,t)\leq \bar{w}(x,t+t_0)\) for all \(t\in[0,T)\).

Recall that Lemma~\ref{lem:concavity} asserts that \(u''(x,t)\leq 0\).  It follows from \(u(x,t)\leq \bar{w}(x,t+t_0)\) and Remark~\ref{rmk:end-point-derivative} that \(0\leq u'(-1,t)\leq \bar{w}'(-1,t+t_0)\) and \(0\geq u'(1,t)\geq \bar{w}'(1,t+t_0)\).  Hence, \(\|u'(x,t)\|_{C^0([-1,1])}\leq \|\bar{w}'(x,t+t_0)\|_{C^0([-1,1])}\).  According to Theorem~\ref{thm:barrier-summary}, \(\|\bar{w}(\,\cdot\,,t)\|_{C^1([-1,1])}\to0\) as \(t\to\infty\), and hence
\(\|u(\,\cdot\,,t)\|_{C^1([-1,1])}\to0\) as  \(t\to\infty\).

We estimate the first three terms on the right-hand side of \eqref{eqn:main-test}:
\begin{align*}
    &\quad -\pl_t g - c_0\frac{u}{\phi} - \frac{2(u')^2}{1+(u')^2} \\
    &\geq \bigl( c_0\sup\frac{1}{\phi} + 2 \bigr)\cdot \bigl( \|\bar{w}(x,t+t_0)\|_{C^0([-1,1])} + \|\bar{w}(x,t+t_0)\|^2_{C^1([-1,1])} \bigr)
    - c_0\frac{u}{\phi} - 2(u')^2 \\
    &\geq 0 ~.
\end{align*}
Now, the maximum principle applies to \(Q(x,t)\) defined by \eqref{fct:main-test}, and \(Q(x,t)\) is always non-negative.  Thus,
\begin{align} \label{eqn:u-double-prime-estimate}
    0\leq -u''(x,t) \leq \frac{u\cdot\phi(x,u,0)}{g(t)}(1+(u')^2) ~.
\end{align}
Since \(g(t)\geq \frac{1}{2\Ld}\), \(y\cdot\phi(x,y,0)\) is uniformly bounded on \(\Om\), and \(u'\) is bounded by the \(C^1\)-norm of \(\bar{w}\), we find that \(u''\) is uniformly bounded on \((-1,1)\times[0,T)\).  According to Lemma~\ref{lem:LMCF-curve-graphical}, \(T\) must be \(\infty\).  The bound \eqref{eqn:condition-later-time} is a direct consequence of \eqref{eqn:u-double-prime-estimate}.

{\textit{Step 3: mean curvature and second fundamental form}}.
By \eqref{eqn:u-double-prime-estimate}
\begin{align}
    |H| &= \frac{1}{\phi^{\frac{1}2}}\frac{-u''}{\big(1+(u')^2\bigr)^{\frac{3}2}} \leq \frac{\sqrt{u\cdot\phi(x,u,0)}}{g(t)} \sqrt{u} \leq c_2\,\|\bar{w}(x,t+t_0)\|^{\frac{1}{2}}_{C^0([-1,1])}
\end{align}
for some \(c_2 > 0\).  It follows that \(H\to 0\) uniformly as \(t\to\infty\).

It remains to analyze the second fundamental form.  According to \cite{Lotay-Oliveira-1}*{Proposition~4.6} and \cite{Lotay-Oliveira-2}*{Proposition~4.1},
\begin{align} \label{eqn:fform-sandwich}
    \frac{1}{c_3}\Bigl(\phi^{-\frac{1}{2}}|u''| + \phi^{-\frac{3}{2}}\bigl|\dd\phi(\bfn)\bigr| + \phi^{-\frac{3}{2}}\bigl|\dd\phi(\pl_z)\bigr|\Bigr) 
    &\leq |A| \leq {c_3}\Bigl(\phi^{-\frac{1}{2}}|u''| + \phi^{-\frac{3}{2}}\bigl|\dd\phi(\bfn)\bigr| + \phi^{-\frac{3}{2}}\bigl|\dd\phi(\pl_z)\bigr| \Bigr)
\end{align}
for some \(c_3>0\), where \(\bfn\) is the unit normal of the plane curve.
From the proof of Lemma~\ref{lem:LMCF-curve-graphical}, both \(\phi^{-\frac{1}{2}}|u''|\) and \(\phi^{-\frac{3}{2}}\bigl|\dd\phi(\pl_z)\bigr|\) remain bounded, and only \(\phi^{-\frac{3}{2}}\dd\phi(\bfn) = -\phi^{\frac{1}{2}}\dd\phi^{-1}(\bfn)\) could blow up as \(t\to\infty\).  Moreover, it is bounded for \(x\in[-1,-1+\dt]\) and \(x\in[1-\dt,1]\), where\footnote{\(a_1 = -1\) and \(a_n = 1\) correspond to smooth points of the infinite-time limit.} \(\dt = \frac{1}{10}\min\{a_2-a_1,a_n-a_{n-1}\}\).

The concavity of \(u\) implies that
\[ u(x,t)\geq \frac{\dt}{2}\|u(\,\cdot\,,t)\|_{C^0([-1,1])} \quad \text{for }x\in[-1+\dt,1-\dt] ~.\]
By the same argument as that in the proof of Lemma~\ref{lem:concavity}, \(|\dd\phi^{-1}|\) is bounded, and thus
\begin{align*}
    \bigl|\phi^{\frac{1}{2}}\dd\phi^{-1}(\bfn)\bigr| &\leq c_4 \phi^{\frac{1}{2}} \leq c_5\bigl(u(x,t)\bigr)^{-\frac{1}{2}}
\end{align*}
for some constants \(c_4,c_5\).  It follows that
\begin{align} \label{eqn:upper-bound-fform-main}
    \sup_{|x|\leq 1-\dt} \bigl|\phi^{\frac{1}{2}}\dd\phi^{-1}(\bfn)\bigr| &\leq c_6\|u(\,\cdot\,,t)\|^{-\frac{1}{2}}_{C^0([-1,1])}
\end{align}
for some constant \(c_6 = c_6(\dt)\).  On the other hand, \(u'(x,t)\to0\) uniformly as \(t\to\infty\), and
\begin{align*}
    \lim_{y\to0^+}(\pl_x\phi^{-1},\pl_y\phi^{-1})(a_i,y,0) &= (0,2)
\end{align*}
for \(i=2,\ldots,n-1\).  It follows that there is a \(c_7 > 0\) such that
\begin{align} \label{eqn:lower-bound-fform-main}
    \bigl|\phi^{\frac{1}{2}}\dd\phi^{-1}(\bfn)\bigr|_{\text{at } x=a_i} &\geq c_7\|u(\,\cdot\,,t)\|^{-\frac{1}{2}}_{C^0([-1,1])}
\end{align}
for \(i=2,\ldots,n-1\).  Putting \eqref{eqn:fform-sandwich}, \eqref{eqn:upper-bound-fform-main} and \eqref{eqn:lower-bound-fform-main} together gives
\begin{align} \label{eqn:fform-sandwich-C0}
    \frac{1}{c_8} \leq \|u(\,\cdot\,,t)\|^{\frac{1}{2}}_{C^0([-1,1])}\cdot\max|A(\,\cdot\,,t)| \leq c_8
\end{align}
for some \(c_8 > 1\).  Since \(\|u(\,\cdot\,,t)\|_{C^0([-1,1])}\) is sandwiched by \(\bar{w}(x,t+t_0)\) and \(\ul{w}(x,t+t_1)\), it follows from \eqref{eqn:fform-sandwich-C0} and Lemma~\ref{lem:C0-asymptotic} that
\begin{align*}
    & \limsup_{t\to\infty}\frac{\log\max|A(\,\cdot\,,t)|}{\sqrt{t}} \leq \frac{1}{2}\lim_{t\to\infty}\frac{\ul{\ld}_{\frac12}(t+t_1)\log\|\ul{w}(\,\cdot\,,t+t_1)\|^{-1}_{C^0}}{\sqrt{t}\cdot\ul{\ld}_{\frac12}(t+t_1)} = \sqrt{\frac{1+\frac12}{2\ul{c}}\mu} \quad\text{and} \\
    & \liminf_{t\to\infty}\frac{\log\max|A(\,\cdot\,,t)|}{\sqrt{t}} \geq \frac{1}{2}\lim_{t\to\infty} \frac{\bar{\ld}_{\frac12}(t+t_0)\log\|\bar{w}(\,\cdot\,,t+t_0)\|^{-1}_{C^0}}{\sqrt{t}\cdot\bar{\ld}_{\frac12}(t+t_0)} = \sqrt{\frac{1-\frac12}{2\bar{c}}\mu} ~.
\end{align*}

{\textit{Step 4: proof of \eqref{eqn:blow-up-rate}}}.
We have already proved the long-time existence, and \(u(\,\cdot\,,t)\to0\) as \(t\to\infty\) in \(C^1\).  For every \(\ep\in(0,\frac12)\), we can find \(t_2\) sufficiently large such that \(|u'(x,t_2)|\leq\frac{1}{c_1(4\Ld,\ep)}\).  Repeating the arguments of Step~2 and Step~3 (with \(u(x,t_2)\) as the initial date) implies that
\begin{align*}
    \sqrt{\frac{(1-\ep)\mu}{2\bar{c}}} \leq \liminf_{t\to\infty}\frac{\log\max|A(\,\cdot\,,t)|}{\sqrt{t}} \leq \limsup_{t\to\infty}\frac{\log\max|A(\,\cdot\,,t)|}{\sqrt{t}} \leq \sqrt{\frac{(1+\ep)\mu}{2\ul{c}}} ~.
\end{align*}
Since this holds true for every \(\ep\in(0,\frac12)\), this completes the proof of the theorem.
\end{proof}

\begin{rmk}
    By \eqref{eqn:hk-metric}, the neck size, namely the circumference of the circle corresponding to \(u(a_i,t)\) for \(i=2,\ldots,n-1\), is of the same order as
    \begin{align*}
        \phi^{-\frac12}{(a_i,u(a_i,t))} &\sim \|u(\,\cdot\,,t)\|^{\frac{1}{2}}_{C^0([-1,1])} \sim \frac{1}{\max|A(\,\cdot\,,t)|} ~.
    \end{align*}
\end{rmk}

Suppose that the ambient hyperk\"ahler manifold is ALE (\(\mass=0\)) or ALF (\(\mass>0\)), and that \(p_1,\cdots,p_n\) are precisely the fixed points of the circle action.  Then \(\Om=\BR^3\) and \(\phi(x,y,0)=\phixaxis(x,y)\); see \cite{Lotay-Oliveira-1}*{Examples~2.4 and~2.5}.  Consequently, \(\bar{c} = 1 = \ul{c}\).

\begin{cor} \label{cor:sharp-rate}
    When the hyperk\"ahler manifold is ALE (\(\mass=0\)) or ALF (\(\mass>0\)), and \(p_1,\ldots,p_n\) are all fixed points of the circle action, assertion \eqref{eqn:blow-up-rate} of Theorem~\ref{thm:main} becomes
    \begin{align}
        \lim_{t\to\infty}\frac{\log\max|A(\,\cdot\,,t)|}{\sqrt{t}} = \sqrt{\frac{\mu}{2}} ~.
    \end{align}
\end{cor}


\end{document}